\documentclass[10pt]{article}
\usepackage{latexsym, amssymb, amsmath, amscd, amsthm, verbatim}
\usepackage{enumerate}
\usepackage[latin1]{inputenc}
\usepackage{amstext}
\usepackage{color}
\usepackage{graphicx}
\usepackage{longtable}
\usepackage{calc}
\usepackage{arydshln}

\numberwithin{equation}{section}
\numberwithin{table}{section}
\numberwithin{figure}{section}
\newtheorem{theorem}{Theorem}[section]

\newtheorem{lemma}[theorem]{Lemma}
\newtheorem{corollary}[theorem]{Corollary}

\newtheorem{definition}[theorem]{Definition}
\theoremstyle{definition}
\newtheorem{example}[theorem]{Example}
\newtheorem{remark}[theorem]{Remark}

\makeatletter
\def\revddots{\mathinner{\mkern1mu\raise\p@
     \vbox{\kern7\p@\hbox{.}}\mkern2mu
     \raise4\p@\hbox{.}\mkern2mu\raise7\p@\hbox{.}\mkern1mu}}
\makeatother

\newcommand {\mat}  [1] {\left[\begin{array}{#1}}
\newcommand {\rix}      {\end{array}\right]}

\newcommand{\la}{\ensuremath{\lambda}}

\def\max{\mathop{\rm max}}
\def\rank{\mathop{\rm rank}}
\def\min{\mathop{\rm min}}

\def\det{\mathop{\rm det}}

\mathchardef\Gamma="7100
\mathchardef\Delta="7101
\mathchardef\Theta="7102
\mathchardef\Lambda="7103
\mathchardef\Xi="7104
\mathchardef\Pi="7105
\mathchardef\Sigma="7106
\mathchardef\Upsilon="7107
\mathchardef\Phi="7108
\mathchardef\Psi="7109
\mathchardef\Omega="710A

\newcommand{\bF}{\mathbb{F}}

\newcommand{\FF}{\mathbb{F}}
\newcommand{\RR}{\mathbb{R}}
\newcommand{\CC}{\mathbb{C}}

\newcommand{\beq}{\begin{equation}}
\newcommand{\eeq}{\end{equation}}

\mathcode`@="8000 
{\catcode`\@=\active\gdef@{\mkern1mu}}

\def\veps{\varepsilon}

\begin{document}

\title{Robustness and perturbations of minimal bases II: \\ The case with given row degrees}
\author{Froil\'an M. Dopico\footnote{Departamento de Matem\'aticas, Universidad Carlos III de Madrid, Avenida de la Universidad 30, 28911, Legan\'es, Spain.  Email: {\tt dopico@math.uc3m.es}. Supported by ``Ministerio de Econom\'{i}a, Industria y Competitividad of Spain'' and ``Fondo Europeo de Desarrollo Regional (FEDER) of EU'' through grants MTM-2015-68805-REDT, MTM-2015-65798-P (MINECO/FEDER, UE).}
\hspace{0.25cm} and \hspace{0.25cm} 
Paul Van Dooren\footnote{Department of Mathematical Engineering, Universit\'e catholique de Louvain, Avenue Georges Lema\^{i}tre 4, B-1348 Louvain-la-Neuve, Belgium.  Email: {\tt paul.vandooren@uclouvain.be}. Supported by the Belgian network DYSCO (Dynamical Systems, Control, and Optimization), funded by the Interuniversity Attraction Poles Programme initiated by the Belgian Science Policy Office.} 
}
\date{\today}
\maketitle

\begin{abstract}
This paper studies generic and perturbation properties inside the linear space of $m\times (m+n)$ polynomial matrices whose rows have degrees bounded by a given list $d_1, \ldots, d_m$ of natural numbers, which in the particular case $d_1 = \cdots = d_m = d$ is just the set of $m\times (m+n)$ polynomial matrices with degree at most $d$. Thus, the results in this paper extend to a much more general setting the results recently obtained in [Van Dooren \& Dopico, Linear Algebra Appl. (2017), {\tt http://dx.doi.org/10.1016/j.laa.2017.05.011}] only for polynomial matrices with degree at most $d$. Surprisingly, most of the properties proved in [Van Dooren \& Dopico, Linear Algebra Appl. (2017)], as well as their proofs, remain to a large extent unchanged in this general setting of row degrees bounded by a list that can be arbitrarily inhomogeneous provided the well-known Sylvester matrices of polynomial matrices are replaced by the new trimmed Sylvester matrices introduced in this paper. The following results are presented, among many others, in this work: (1) generically the polynomial matrices in the considered set are minimal bases with their row degrees exactly equal to $d_1, \ldots , d_m$, and with right minimal indices differing at most by one and having a sum equal to $\sum_{i=1}^{m} d_i$, and (2), under perturbations, these generic minimal bases are robust and their dual minimal bases can be chosen to vary smoothly.
\end{abstract}

{\small \noindent
 }


\section{Introduction} \label{sect.intro}
Minimal bases of rational vector spaces, usually arranged as the rows of polynomial matrices, are a standard tool in control theory and in coding theory. Therefore, their definition, properties, and many of their practical applications can be found in classical references on these subjects, as, for instance, the ones by Wolovich \cite{wolovich}, Kailath \cite{Kai80}, and Forney \cite{For75}, although the concept of minimal bases is much older and, as far as we know, it was introduced for the first time in the famous paper by Dedekind and Weber \cite{dedekind}. Recently, minimal bases, and the closely related notion of pairs of dual minimal bases, have been applied to some problems that have attracted considerable attention in the last years as, for instance, in the solution of inverse complete eigenstructure problems for polynomial matrices \cite{DDMV,DDV}, in the development of new classes of linearizations and $\ell$-ifications of polynomial matrices \cite{DDV-l-ifications,blockKron,lawrence-perez-cheby,robol},
 in the explicit construction of linearizations of rational matrices \cite{amparanjointrat}, and in the backward error analysis of complete polynomial eigenvalue problems solved via different classes of linearizations \cite{blockKron,lawrence-vanbarel-vandooren}.

Some of the applications mentioned in the previous paragraph motivated the development in the recent paper \cite{VD} of robustness and perturbation results of minimal bases, which had not been explored before in the literature. The study of any perturbation problem for polynomial matrices requires as a first step to fix the set of allowable perturbations and, with this purpose, the reference \cite{VD} considers perturbations whose only constraint is that they do not increase the degree $d$ of the $m \times (m+n)$ given minimal basis that is perturbed. These perturbations are certainly natural, very mild, and, moreover, cover the main applications that are mentioned in \cite{VD}, that is,  backward error analyses of algorithms for solving polynomial eigenvalue problems with linearizations. The perturbation theory in \cite{VD} is based on a number of new results that were also obtained in \cite{VD} as, for instance, a new characterization of minimal bases in terms of their Sylvester matrices. Moreover, \cite{VD} establishes that the polynomial matrices of size $m \times (m+n)$ and degree at most $d$ are generically minimal bases with the degrees of all their rows equal to $d$ and with their right minimal indices satisfying the following two key properties: they are ``almost homogeneous'', i.e., they differ at most by one, and their sum is equal to $d m$. The perturbation results in \cite{VD} are only valid for these generic minimal bases, which are ``highly homogeneous'' from the perspectives mentioned above.

In order to describe sets of polynomial matrices with bounded rank and degree in a more explicit way than the one presented in \cite{dmy-dop-2017}, one needs to consider perturbations of a minimal basis $M(\la)$ with much stronger constraints than the one imposed in \cite{VD}, since such perturbations cannot increase the individual degree of each of the rows of $M(\la)$. More precisely, given an $m \times (m+n)$ minimal basis $M(\la)$ whose rows have degrees $d_1, d_2, \ldots , d_m$, which can be arbitrarily different each other, or, in other words, ``arbitrarily inhomogeneous'', the perturbed polynomial matrix $M (\la) + \Delta M(\la)$ must have rows with degrees at most $d_1, d_2, \ldots , d_m$. These perturbations must stay in the set of $m \times (m+n)$ polynomial matrices whose rows have degrees at most $d_1, d_2, \ldots , d_m$ and, therefore, this is the set studied in this paper. It is clear that the polynomial matrices in this set have generically rows with degrees exactly equal to $d_1, d_2, \ldots , d_m$ and, so, are very different from the generic polynomial matrices arising in \cite{VD}, which have the degrees of their rows all equal to $d$, that is, completely homogeneous. Despite this important difference, the results presented in this work are to a large extent similar to those in \cite{VD}, which at a first glance is rather surprising. Thus, we prove in this paper that the polynomial matrices of size $m \times (m+n)$ and with the degrees of their rows bounded by $d_1, d_2, \ldots , d_m$ are generically minimal bases with the degrees of their rows exactly equal to $d_1, d_2, \ldots , d_m$ and having ``almost homogeneous'' right minimal indices with sum equal to $\sum_{i=1}^{m} d_i$. We remark that this ``almost homogeneity'' of the right minimal indices, or, equivalently, of the degrees of the dual minimal bases, is the key property that allows us to develop a perturbation theory analogous to the one presented in \cite{VD}. In order to prove these new results, we need to introduce a new tool that, although simple, we think it has not been used before in the literature: the trimmed Sylvester matrices associated with a polynomial matrix. Once this new tool and its properties are derived, most of the proofs in this paper are rather similar to those in \cite{VD} and only the relevant differences will be discussed.

Among all the results presented in this paper, perhaps the most remarkable one is the genericity of the property of ``almost homogeneous'' minimal indices in a set of polynomial matrices whose elements have generically rows with inhomogeneous degrees. We think that this is the first time that this phenomenon has been observed, since, until now, the genericity of ``almost homogeneous'' minimal indices is a well-known fact that has been proved only in scenarios where the generic situation is that the degrees of the rows are all equal. Thus, in the case of pencils, i.e., polynomial matrices with degree at most one, this property was observed for the first time in \cite{vandoorenphd} and, then, in many other references from different perspectives as, for instance, in \cite{Bole98,DeEd95,EdEK97,EdEK99} for general pencils and in \cite{DeDo07,DeDo08} for pencils satisfying certain properties. In the case of polynomial matrices with degree at most $d$, where $d$ is an arbitrary positive integer, results on this generic property are much more recent and can be extracted from the general stratification results in \cite{DJKV15,JoKV13} and are explicitly stated in \cite{dmy-dop-2017,VD}.

As said before, the results and proofs in this paper are closely connected to those in \cite{VD} and, so, we will refer as much as possible to that paper for the proofs that can be found there. Nevertheless, for the sake of readability of the current paper, we will repeat here some definitions and crucial theorems that are needed to understand the new results. This said, the paper is organized as follows. In Section \ref{Sec:Preliminaries} we recall some basic definitions and properties of minimal bases and introduce some concepts and notations to be used throughout the paper. In Section \ref{Sec:Trimmed} we introduce trimmed Sylvester matrices and show that have similar properties as the classical Sylvester matrices. Section \ref{Sec:minbases-from-Trimmed} establishes general properties of the minimal bases that belong to the set of $m\times (m + n)$ polynomial matrices whose rows have degrees bounded by a list of numbers.
In Section \ref{Sec:FullRank} we introduce full-trimmed-Sylvester-rank matrices and show that they correspond to minimal bases with almost homogeneous right minimal indices. In Sections
\ref{Sec:Genericity} and \ref{Sec:Smoothness} we then show that these full-trimmed-Sylvester-rank matrices also correspond to the generic situation and that these minimal bases therefore also have good robustness properties. The perturbations of their dual minimal bases are then analyzed in Section \ref{Sec:Dualminbases}. Finally, we revisit the classical conditions for minimal bases and show in Section \ref{Sec:Classrevisited} that the robustness of these conditions follows from the robustness properties of Section \ref{Sec:Smoothness}.
In the concluding Section \ref{Sec:Conclusions} we summarize the main results presented in the paper and discuss some of their possible applications. Except for Section \ref{Sec:Trimmed}, the remaining sections of this paper are counterparts of sections in \cite{VD} and the specific relationships will be commented in each section. However, \cite[Section 9]{VD} has no counterpart here since it is based on a property that is not preserved for minimal bases with inhomogeneous row degrees: their reversal polynomial matrices are no longer minimal bases.

\section{General preliminaries and the space $\FF[\la]^{m \times (m+n)}_{\underline{d}}$}
\label{Sec:Preliminaries}
This section introduces notations, nomenclature, and basic concepts used in the rest of the paper. The first part of the section is a summary of \cite[Section 2]{VD} and is included for convenience of the reader, who can find more complete information in \cite{VD}. The second part introduces the vector space $\FF[\la]^{m \times (m+n)}_{\underline{d}}$ of polynomial matrices whose rows have degrees bounded by the elements of a given list $\underline{d}$, which is the space containing the polynomials studied in this paper. In addition, some basic properties of $\FF[\la]^{m \times (m+n)}_{\underline{d}}$ are established.

The results in Sections \ref{Sec:Trimmed}, \ref{Sec:minbases-from-Trimmed}, and \ref{Sec:FullRank} of this paper hold in any field $\FF$, while in the remaining sections $\FF$ is the field of real numbers $\RR$ or of complex numbers $\CC$.
We adopt standard notation: $\FF[\la]$ denotes the ring of polynomials in the variable $\la$ with coefficients in $\FF$ and $\FF(\la)$ denotes the field of fractions of $\FF[\la]$. Vectors with entries in $\FF[\la]$ are called polynomial vectors. In addition, $\FF[\la]^{m\times n}$ stands for the set of $m\times n$ polynomial matrices, and $\FF(\la)^{m\times n}$ for the set of $m\times n$ rational matrices. The {\em degree} of a polynomial vector, $v(\la)$, or matrix, $P(\la)$, is the highest degree of all its entries and is denoted by $\deg(v)$ or $\deg(P)$. Finally, $\overline \FF$ denotes the algebraic closure of $\FF$, $I_n$ the $n\times n$ identity matrix,
and $0_{m\times n}$ the $m\times n$ zero matrix, where the sizes are omitted when they are clear from the context.

The rank of $P(\la)$ (sometimes called ``normal rank'') is just the rank of $P(\la)$ considered as a matrix over the field $\FF (\la)$,
and is denoted by $\mbox{rank} (P)$. Other concepts on polynomial matrices used in this paper can be found in the classical books \cite{Gan59,Kai80}, as well as in the summary included in \cite[Sect. 2]{DDM}.

Since ``minimal basis'' is the key concept of this paper, we revise its definition and characterization. It is well known that every rational vector subspace $\mathcal{V}$,
i.e., every subspace $\mathcal{V} \subseteq \FF(\la)^n$ over the field $\FF(\la)$,
has bases consisting entirely of polynomial vectors. Among them some are minimal in the following sense introduced by Forney \cite{For75}: a {\em minimal basis} of $\mathcal{V}$ is a basis of $\mathcal{V}$ consisting of polynomial vectors whose sum of degrees is minimal among all bases of $\mathcal{V}$  consisting of polynomial vectors. The fundamental property \cite{For75,Kai80} of such bases is that the ordered list of degrees of the polynomial vectors in any minimal basis of $\mathcal{V}$ is always the same. Therefore, these degrees are an intrinsic property of the subspace $\mathcal{V}$ and are called the {\em minimal indices} of $\mathcal{V}$.
This discussion leads us to the definition of the minimal bases and indices of a polynomial matrix. An $m\times n$ polynomial matrix $P(\lambda)$ with rank $r$ smaller than $m$ and/or $n$ has non-trivial left and/or right rational null-spaces, respectively,
over the field $\FF (\la)$, which are denoted by ${\cal N}_{\ell}(P)$ and ${\cal N}_r(P)$, respectively.
Polynomial matrices with non-trivial ${\cal N}_{\ell}(P)$ and/or ${\cal N}_r(P)$
are called {\em singular} polynomial matrices. If the rational subspace ${\cal N}_{\ell}(P)$ is non-trivial, it has minimal bases and minimal indices, which are called the {\em left minimal bases and indices} of $P(\la)$. Analogously, the {\em right minimal bases and indices} of $P(\la)$ are those of ${\cal N}_{r}(P)$, whenever this subspace is non-trivial.

The definition of minimal basis given above cannot be easily handled in practice. Therefore, we include in Theorem \ref{minbasis_th} a classical characterization introduced in \cite[p. 495]{For75} that is more useful, although it requires to check the ranks of infinitely many constant matrices. We emphasize that, recently, a characterization in terms of the ranks of a finite number of constant matrices has been obtained in \cite[Theorem 3.7]{VD} and that this other characterization is refreshed later in Theorem \ref{th:finitenumbranks}. The statement of Theorem \ref{minbasis_th} requires to introduce Definition \ref{colred}.
For brevity, we use the expression ``column (resp., row) degrees'' of a polynomial matrix to mean the degrees of its column (resp., row) vectors.

\begin{definition}\label{colred}
Let $d'_1,\ldots,d'_n$ be the column degrees of $N(\lambda) \in \FF[\la]^{m\times n}$.
The highest-column-degree coefficient matrix of $N(\la)$, denoted by $N_{hc}$,
is the $m\times n$ constant matrix whose {\rm $j$th} column
is the vector coefficient of $\lambda^{d'_j}$ in the {\rm $j$th} column of $N(\lambda)$.
The polynomial matrix $N(\la)$ is said to be \emph{column reduced} if $N_{hc}$ has full column rank.

Similarly, let $d_1,\ldots,d_m$ be the row degrees of $M(\la) \in \FF[\la]^{m\times n}$. The highest-row-degree coefficient matrix of $M(\la)$, denoted by $M_{hr}$,
is the $m\times n$ constant matrix whose {\rm $j$th} row is the vector coefficient
of $\lambda^{d_j}$ in the {\rm $j$th} row of $M(\lambda)$.
The polynomial matrix $M(\la)$ is said to be \emph{row reduced} if $M_{hr}$ has full row rank.
\end{definition}

Theorem \ref{minbasis_th} provides the announced characterization of minimal bases proved in \cite{For75}.

\begin{theorem}\label{minbasis_th}
  The columns {\rm (}resp., rows{\rm )} of a polynomial matrix $N(\lambda)$ over a field $\FF$
  are a minimal basis of the subspace they span
  if and only if $N(\lambda_0)$ has full column {\rm (}resp., row{\rm )} rank
  for all $\lambda_0 \in \overline \FF$,
  and $N(\la)$ is column {\rm (}resp., row{\rm )} reduced.
\end{theorem}

\begin{remark} \label{rem:NEWconvention}
In this paper we follow the convention in \cite{For75} and often say, for brevity, that a $p \times q$ polynomial matrix $N(\la)$ is a {\em minimal basis} if the columns (when $q<p$) or rows (when $p<q$) of $N(\la)$ form a minimal basis of the rational subspace they span. Most of the minimal bases considered in this paper are arranged as the rows of matrices. Recall also that if $M(\la)\in \bF[\lambda]^{m\times k}$ is a row (resp. column) reduced polynomial matrix, then $M(\la)$ has full row (resp. column) (normal) rank.
\end{remark}

Next, we introduce the concept of {\em dual minimal bases}, whose origins can be found in \cite[Section 6]{For75} and that has played a key role in a number of recent applications (see \cite{VD} and \cite{DDMV} for more information).
\begin{definition}\label{dual-m-b}
  Polynomial matrices $M(\lambda)\in \bF[\lambda]^{m\times k}$
  and $N(\lambda)\in \bF[\lambda]^{n\times k}$ with full row ranks
  are said to be {\em dual minimal bases}
  if they are minimal bases satisfying $m+n=k$ and $M(\lambda) \, N(\lambda)^T=0$.
\end{definition}
In the language of null-spaces of polynomial matrices, observe that $M(\lambda)$ is a minimal basis of ${\cal N}_\ell (N(\la)^T)$
and that $N(\lambda)^T$ is a minimal basis of ${\cal N}_r (M(\la))$.
As a consequence, the right minimal indices of $M(\la)$ are the row degrees of $N(\la)$ and the left minimal indices of $N(\la)^T$ are the row degrees of $M(\la)$.

The next theorem reveals a fundamental relationship between the row degrees of dual minimal bases. Its first part was proven in \cite{For75}, while the second (converse) part has been proven very recently in \cite{DDMV}.
\begin{theorem} \label{thm:dualbasis}
  Let $M(\la)\in \bF[\lambda]^{m\times (m + n)}$ and $N(\lambda) \in \bF[\lambda]^{n\times (m+n)}$
  be dual minimal bases with row degrees $(\eta_1,\hdots,\eta_m )$
  and $(\varepsilon_1,\hdots,\veps_{n} )$, respectively.
  Then
\begin{equation} \label{eqn.keyequality}
  \sum_{i=1}^m \eta_i \,=\, \sum_{j=1}^{n}\varepsilon_j \,.
\end{equation}
Conversely, given any two lists of nonnegative integers
$(\eta_1,\hdots,\eta_m)$ and $(\varepsilon_1,\hdots,\veps_{n})$
satisfying \eqref{eqn.keyequality},
there exists a pair of dual minimal bases $M(\la)\in \bF[\lambda]^{m\times (m+n)}$
and $N(\lambda) \in \bF[\lambda]^{n\times (m+n)}$
with precisely these row degrees, respectively.
\end{theorem}

This paper studies the set of polynomial matrices of size $m \times (m + n)$ and with row degrees at most $d_1, d_2, \ldots , d_m$, where $d_1, d_2, \ldots , d_m$ are given nonnegative integers which are stored in the list $\underline{d} := (d_1, d_2, \ldots , \allowbreak d_m)$. This set is formally defined as
\begin{equation} \label{eq.Fdbar}
\FF[\la]^{m \times (m+n)}_{\underline{d}} :=
\left\{
\left[\begin{array}{c} R_1(\lambda) \\ R_2(\lambda) \\ \vdots \\  R_m(\lambda)\end{array}\right] \, : \, \begin{array}{l}
R_i(\la) =R_{i,0}+ R_{i,1} \la + \cdots +  R_{i,d_i} \la^{d_i}, \\ R_{i,j}\in \FF^{1 \times (m+n)}, \quad 1 \leq i \leq m, \; 0 \leq j \leq d_i
\end{array}
\right\} \, .
\end{equation}
Observe that $\FF[\la]^{m \times (m+n)}_{\underline{d}}$ is a linear space  over the field $\mathbb{F}$ of dimension $(m+n)\sum_{i=1}^m (d_i+1)$. Next, we define, attached to each matrix in $\FF[\la]^{m \times (m+n)}_{\underline{d}}$, a constant matrix that will be often used in this paper.
\begin{definition} \label{def.Mbard} Let $M(\la) \in \FF[\la]^{m \times (m+n)}_{\underline{d}}$, where $\underline{d} = (d_1, d_2, \ldots , d_m)$, and let $R_{i,d_i} \in \FF^{1 \times (m+n)}$ be the vector coefficient of $\la^{d_i}$ in the {\rm $i$th} row of $M(\la)$ for $i =1,\ldots, m$. The leading row-wise coefficient matrix of $M(\la)$ is defined as
\[
M_{\underline{d}}:= \left[\begin{array}{c} R_{1,d_1} \\  R_{2,d_2} \\ \vdots \\  R_{m,d_m} \end{array}\right] \in \FF^{m \times (m+n)}.
\]
\end{definition} \noindent
Note that $M_{\underline{d}}$ is in general different from the highest-row-degree coefficient matrix $M_{hr}$ of $M (\la)$ introduced in Definition \ref{colred}. This is related to the linear space structure of $\FF[\la]^{m \times (m+n)}_{\underline{d}}$ and is emphasized in the next simple lemma whose trivial proof is omitted.
\begin{lemma} \label{lemm.coeffmatrix}  Let $M(\la) \in \FF[\la]^{m \times (m+n)}_{\underline{d}}$, where $\underline{d} = (d_1, d_2, \ldots , d_m)$, and let $R_{i} (\la)$ be the {\rm $i$th} row of $M(\la)$ for $i =1,\ldots, m$. Then:
\begin{itemize}
  \item[\rm (a)] $M_{\underline{d}} = M_{hr}$ if and only if $\deg(R_{i} (\la)) = d_i$ for $i = 1,\ldots, m$.
  \item[\rm (b)]  If $\rank (M_{\underline{d}}) = m$, then $M_{\underline{d}} = M_{hr}$.
\end{itemize}
\end{lemma}
\noindent Part (b) of Lemma \ref{lemm.coeffmatrix} follows from part (a) because $\rank (M_{\underline{d}}) = m$ implies that $R_{i,d_i} \ne 0$ for $i =1,\ldots, m$.

The linear space of polynomial matrices of size $m \times (m + n)$ and degree at most $d$ is also used in this paper and is denoted and defined as follows:
\begin{equation}\label{def.largelinsapce}
\FF[\la]^{m \times (m+n)}_{d} := \left\{ M_0 + M_1 \la + \cdots +  M_d \la^d \, : \, M_i\in \FF^{m\times (m+n)}, \, 0 \leq i \leq d \right\} \, .
\end{equation}
The set $\FF[\la]^{m \times (m+n)}_d $ is also a linear space over $\FF$ and its dimension is $m(m+n)(d+1)$. Given any $M(\la) = M_0 + M_1 \la + \cdots +  M_d \la^d \in \FF[\la]^{m \times (m+n)}_{d}$, the matrix $M_d$ is called the leading coefficient matrix of $M(\la)$. In the case $d = \max_{1 \leq i \leq m} d_i$, it is clear that $\FF[\la]^{m \times (m+n)}_{\underline{d}} \subseteq \FF[\la]^{m \times (m+n)}_{d}$, with equality if and only if $d = d_1 = \cdots = d_m$. Therefore, $\FF[\la]^{m \times (m+n)}_{\underline{d}}$ is a linear subspace of $\FF[\la]^{m \times (m+n)}_{d}$. Throughout the paper we will assume that $m > 0$, $n > 0$, and $d >0$ for avoiding trivialities.

Finally, we illustrate with an example the differences among $M_{\underline{d}}, M_{hr}$, and $M_d$.
\begin{example} Let $m= 3$, $n=1$, $\underline{d} = (1, 3, 10)$, and $d=10$. Consider
\[
M(\la) = \left[ \begin{array}{cccc}
                  1 & 0 & \la + 1 & \la \\
                  \la^2 & 2 & \la & 2 \la^2 \\
                  -1 & \la^{10} - \la^9 & 1 & \la^8
                \end{array}\right] \in \FF[\la]^{3 \times 4}_{\underline{d}} \, .
\]
Then,
\[
M_{hr} = \left[ \begin{array}{cccc}
                  0 & 0 & 1 & 1 \\
                  1 & 0 & 0 & 2 \\
                  0 & 1 & 0 & 0
                \end{array}\right], \quad
M_{\underline{d}} = \left[ \begin{array}{cccc}
                  0 & 0 & 1 & 1 \\
                  0 & 0 & 0 & 0 \\
                  0 & 1 & 0 & 0
                \end{array}\right], \quad
M_d = \left[ \begin{array}{cccc}
                  0 & 0 & 0 & 0 \\
                  0 & 0 & 0 & 0 \\
                  0 & 1 & 0 & 0
                \end{array}\right] \, .
\]
So, in this case, the three matrices are different. If we had considered $\underline{d} = (1, 2, 10)$, then we would have obtained $M_{\underline{d}} = M_{hr}$.
\end{example}
It is interesting to remark, at the light of the previous example, that given a polynomial matrix $M(\la)$, the matrix $M_{hr}$ is intrinsically attached to $M(\la)$, while $M_{\underline{d}}$ and $M_d$ vary with the list $\underline{d}$  and the value $d$ that are considered, i.e., with the sets containing $M(\la)$ that are considered.

\section{Trimmed Sylvester matrices of polynomials in $\FF[\la]^{m \times (m+n)}_{\underline{d}}$} \label{Sec:Trimmed} In this section, we introduce and study certain constant matrices attached to the polynomial matrices in $\FF[\la]^{m \times (m+n)}_{\underline{d}}$. These matrices are called {\em trimmed Sylvester matrices} and are essential for obtaining the results in this paper. They are built from the Sylvester matrices \cite{AndJuryIEEE76,BKAK} associated to the polynomial matrices in $\FF[\la]^{m \times (m+n)}_{d}$, which were heavily used in \cite{VD} and whose definition is refreshed below.

\begin{definition} \label{def.Sylmatrices} Let $M(\la) = M_0 + M_1 \la + \cdots +  M_d \la^d \in \FF[\la]^{m \times (m+n)}_{d}$. The {\rm $k$th} Sylvester matrix of $M(\la)$ is defined as
\begin{equation}\label{eq:Sylvester}
   S_k(M) := \underbrace{\left[ \begin{array}{ccccc} M_0 \\ M_1 & M_0 \\
   \vdots & M_1 &  \ddots \\
   M_d & \vdots & \ddots & M_0 \\
   0 & M_d & & M_1\\
   \vdots & \ddots & \ddots & \vdots \\
   0 & \ldots & 0 & M_d
   \end{array}\right]}_{\mbox{\rm $k$ block columns}} \in  \FF^{(k+d)m\times k(m+n)} \, .
  \end{equation}
\end{definition}
\noindent When it is obvious from the context, we will drop the argument $(M)$ and just use $S_k$ for denoting the $k$th Sylvester matrix of $M(\la)$.

The Sylvester matrices of those $M(\la) \in \FF[\la]^{m \times (m+n)}_{\underline{d}} \subseteq \FF[\la]^{m \times (m+n)}_{d}$, where $\underline{d} = (d_1, d_2, \ldots , d_m)$ and $d = \max_{1 \leq i \leq m} d_i$, have several rows that are zero {\em for any} $M(\la) \in \FF[\la]^{m \times (m+n)}_{\underline{d}}$. These zero rows are identified in the next lemma, whose simple proof is omitted.

\begin{lemma} \label{lemm.zerorows} Let $M(\la) \in \FF[\la]^{m \times (m+n)}_{\underline{d}} \subseteq \FF[\la]^{m \times (m+n)}_{d}$, where $\underline{d} = (d_1, d_2, \ldots , d_m)$ and $d = \max_{1 \leq i \leq m} d_i$, and let $R_{i,j} \in \FF^{1 \times (m+n)}$ be the vector coefficient of $\la^j$ in the {\rm $i$th} row $R_i (\la)$ of $M(\la)$, for $i =1,\ldots, m$ and $j = 0, 1, \ldots, d_i$, as in \eqref{eq.Fdbar}. Then the submatrix of $S_k (M)$ that selects the {\rm $i$th} row of each of the $(k+d)$ block rows of $S_k (M)$ is
\begin{equation}\label{eq:SylvesterR}
   S_k(R_i) := \underbrace{\left[ \begin{array}{ccccc} R_{i,0} \\ R_{i,1} & R_{i,0} \\
   \vdots & R_{i,1} &  \ddots \\
   R_{i,d_i} &\vdots & \ddots & R_{i,0} \\
   0_{1 \times (m+n)} & R_{i,d_i} & & R_{i,1}\\
   \vdots & 0_{1 \times (m+n)} & \ddots & \vdots \\
   \vdots &  & \ddots & R_{i,d_i} \\
   0 & \ldots & \ldots & 0_{(d-d_i)\times (m+n)}
   \end{array}\right]}_{\mbox{\rm $k$ block columns}} \in  \FF^{(k+d)\times k(m+n)} \, ,
  \end{equation}
where the definition of $S_k(R_i)$ assumes that $R_i (\la) \in \FF[\la]_d^{1 \times (m+n)}$.
\end{lemma}

Observe that Lemma \ref{lemm.zerorows} identifies $(d-d_i)$ zero rows at the bottom of each $S_k(R_i)$ and, so, a total of $md - \sum_{i=1}^{m} d_i$ zero rows in $S_k(M)$ {\em for any} $M(\la) \in \FF[\la]^{m \times (m+n)}_{\underline{d}}$. We emphasize, first, that the number of these zero rows is independent of $k$ and, second, that for a particular $M(\la) \in \FF[\la]^{m \times (m+n)}_{\underline{d}}$, the matrix $S_k(R_i)$ can have more zero rows and, so, the same happens for $S_k(M)$. Such additional zero rows appear, for instance, if $R_{i,0}=0$ or if $R_{i,d_i}=0$ or if $R_{i,j}=0$ for all $0 \leq j \leq (k-1)$, for some $i$, as well as in other situations. Since the rows of $S_k(M)$ that are zero for all the elements of $\FF[\la]^{m \times (m+n)}_{\underline{d}}$ do not carry any information on the polynomial matrices of this set, we can remove them, which leaves $k+d_i$ rows coming from each $S_k(R_i)$ and a total of $km+\sum_{i=1}^m d_i$ rows coming from the whole $S_k(M)$. This process leads to the definition of the key constant matrices used in this paper.

\begin{definition} \label{def.trimSylmatrices} Let $M(\la) \in \FF[\la]^{m \times (m+n)}_{\underline{d}} \subseteq \FF[\la]^{m \times (m+n)}_{d}$, where $\underline{d} = (d_1, d_2, \ldots , d_m)$ and $d = \max_{1 \leq i \leq m} d_i$, and let $S_k (M)$ and $S_k(R_i)$ be, respectively, the {\rm $k$th} Sylvester matrices of $M(\la)$ and of the {\rm $i$th} row of $M(\la)$. The {\rm $k$th} {\rm trimmed Sylvester matrix} of $M(\la)$ is the submatrix of $S_k (M)$ obtained by removing, for $i=1,\ldots ,m$, the $(d-d_i)$ zero rows at the bottom of the submatrix of $S_k (M)$ corresponding to $S_k(R_i)$. The {\rm $k$th} trimmed Sylvester matrix of $M(\la)$ is denoted as
$$T_k(M)\in  \FF^{(km+\sum_{i=1}^m d_i)\times k(m+n)}.$$
\end{definition}

As in the case of Sylvester matrices, we will drop the argument $(M)$ and just use $T_k$ for denoting the $k$th trimmed Sylvester matrix of $M(\la)$, when it is obvious from the context.

Trimmed Sylvester matrices satisfy the structural nesting property that is shown in Lemma \ref{lemm.nesttrim}. This nesting property differs from the one of Sylvester matrices that is displayed in the proof of \cite[Lemma 3.2]{VD} in two aspects: first, it requires the use of a permutation and, second, it involves the matrix $M_{\underline{d}}$ introduced in Definition \ref{def.Mbard}, instead of the matrix $M_d$ that appears in Sylvester matrices. Nevertheless, this nesting property will allow us to prove for the polynomial matrices in $\FF[\la]^{m \times (m+n)}_{\underline{d}}$ properties analogous to those proved in \cite{VD} for the matrices in $\FF[\la]^{m \times (m+n)}_{d}$ using the same techniques, with changes just coming from the use of $M_{\underline{d}}$ and from the fact that $T_k(M)$ and $S_k(M)$ have different sizes. In particular, we will prove in Section \ref{Sec:Genericity} that the right minimal indices of the matrices in $\FF[\la]^{m \times (m+n)}_{\underline{d}}$ are generically ``almost homogeneous'' (i.e., they differ at most by one), as also happens in $\FF[\la]^{m \times (m+n)}_{d}$. Such ``almost homogeneity''  may seem surprising at a first glance, since the row degrees of the matrices in $\FF[\la]^{m \times (m+n)}_{\underline{d}}$ are generically equal to the entries of $\underline{d}$ and, so, they are extremely unbalanced if the entries of $\underline{d}$ are, in contrast with those of the matrices in $\FF[\la]^{m \times (m+n)}_d$, which are generically all equal to $d$.

\begin{lemma} \label{lemm.nesttrim}
Let $M(\la) \in \FF[\la]^{m \times (m+n)}_{\underline{d}}$, let $T_k$ be the trimmed Sylvester matrices of $M(\la)$ for $k=1,2,\ldots$, and let $M_{\underline{d}}$ be the leading row-wise coefficient matrix of $M(\la)$ as in Definition \ref{def.Mbard}. Then, there exist permutation matrices $P_k$, for $k=1,2, \ldots$, such that
$$
P_{1} \, T_{1} = \left[\begin{array}{c} X_1 \\ M_{\underline{d}} \end{array}\right] \quad \mbox{and} \quad
P_{k+1} \, T_{k+1} = \left[\begin{array}{c|c} T_k & X_{k+1} \\  \hline 0 & M_{\underline{d}} \end{array}\right] \quad \mbox{for $k=1,2,...$}.
$$
Moreover, $P_k$ depends on $k$ and $\underline{d}$ but not on $M(\la)$.
\end{lemma}

\begin{proof}
The existence of $P_1$ satisfying the first equality is obvious because $T_1$ contains the rows of $M_{\underline{d}}$. For the second equality, note that the way in which $T_{k+1}$ is obtained by removing zero rows from the Sylvester matrix $S_{k+1}$ of $M(\la)$ and equation \eqref{eq:SylvesterR} guarantee that
\[
   T_{k+1}(R_i) = \left[ \begin{array}{ccc|c} R_{i,0} &&&\\ R_{i,1} & R_{i,0} &&\\
   \vdots & R_{i,1} &  \ddots \\
   R_{i,d_i} & \vdots & \ddots & R_{i,0} \\
   0 & R_{i,d_i} & & R_{i,1}\\
   \vdots & \ddots & \ddots & \vdots \\ \hline
   0 & \hdots  & 0 & R_{i,d_i}
   \end{array}\right] =  \left[ \begin{array}{c|c} T_{k}(R_i) & X_{k+1}^{(i)} \\ \hline 0 &  R_{i,d_i} \end{array}\right] \,
\]
is a submatrix of $T_{k+1}$ for $i=1,\ldots , m$. If $P_{k+1}$ is the permutation matrix that moves the rows of $T_{k+1}$ corresponding to the last row of each of its submatrices $T_{k+1}(R_i)$ to the $m$ bottom positions, then $P_{k+1} T_{k+1}$ has the desired expression.
\end{proof}

Next, we illustrate the definition of trimmed Sylvester matrices and the nesting structure revealed in Lemma \ref{lemm.nesttrim} with two examples. The first one is symbolical and the second one numerical.

\begin{example} \label{ex.1symbolc} In this example, we consider $m = 3$ and $\underline{d} = (0,1,2)$, i.e., $d_1=0$, $d_2=1$ and $d_3=2$, and write explicitly the 3rd trimmed Sylvester matrix of any $M(\la) \in \FF[\la]^{3 \times (3 + n)}_{\underline{d}}$, using the notation in \eqref{eq.Fdbar}, as well as the nesting structure of $P_3 T_3$ involving $T_2$:
\[
T_3 = \left[ \begin{array}{c|c|c} R_{1,0} & & \\ R_{2,0} & & \\ R_{3,0} & & \\ \hline
0 & R_{1,0} \\    R_{2,1} & R_{2,0} \\   R_{3,1} & R_{3,0} \\ \hline
0 & 0 & R_{1,0} \\  0 &  R_{2,1} & R_{2,0} \\ R_{3,2} & R_{3,1} & R_{3,0} \\ \hline
& 0 &  R_{2,1} \\  & R_{3,2} & R_{3,1} \\ \hline & & R_{3,2} \\
\end{array}\right], \quad
P_3  T_3 = \left[ \begin{array}{cc|c} R_{1,0} & & \\ R_{2,0} & & \\ R_{3,0} & & \\
0 & R_{1,0} \\    R_{2,1} & R_{2,0} \\   R_{3,1} & R_{3,0} \\
0 &  R_{2,1} & R_{2,0} \\ R_{3,2} & R_{3,1} & R_{3,0} \\
& R_{3,2} & R_{3,1} \\  \hline
&  & R_{1,0} \\   &  &  R_{2,1} \\  & & R_{3,2} \\
\end{array}\right] =  \left[ \begin{array}{c|c}  T_2 & X_3 \\  \hline 0 & M_{\underline{d}}\end{array}\right] .
\]
The lines partitioning $T_3$ correspond to the partition of the Sylvester matrix $S_3$ and show that the last two block rows of $S_3$ were ``trimmed'' to get $T_3$, since all the block rows of $S_3$ have $3$ rows and the last two block rows of $T_3$ displayed above by the lines have $2$ rows and $1$ row, respectively.
\end{example}

\begin{example} \label{ex:1} In this example, we take $m=4$, $n=3$, $\underline{d} = (0,1,1,2)$, i.e., $d_1=0$, $d_2=d_3=1$ and $d_4=2$, and consider the following matrix $M(\la)\in  \FF[\la]^{4\times 7}_{\underline{d}}$:
	$$M(\la) = \left[ \begin{array}{ccccccc} 1 & 0 & & & &\\  & & -1 & \la & 0 & & \\ & & 0 & -1 & \la & & \\ & & & & & -1 & \la^2 \end{array} \right] .
	$$
	It can be easily checked that the 2nd trimmed Sylvester matrix of $M(\la)$ is given by
	$$T_2 = \left[ \begin{array}{ccccccc|ccccccc} 1 & 0 & & & & & \\ & & -1 & 0 & 0 & & \\ & & 0 & -1 & 0 & & \\ & & & & & -1 & 0 \\ \hline
	&  &&&&&& {\bf 1} & {\bf 0} & & & & & \\ && 0 & 1 & 0 &&& & & -1 & 0 & 0 & & \\ && 0 & 0 & 1 &&& & & 0 & -1 & 0 & & \\ &&&&& 0 & 0 & & & & & & -1 & 0 \\ \hline
	&&&&&&&  & & {\bf 0} & {\bf 1} & {\bf 0}  & & \\
	&&&&&&&  & & {\bf 0} & {\bf 0} & {\bf 1} & & \\
	&&&&& 0 & 1 &  & & & & &  0 & 0  \\ \hline
	&&&&&&& &&&&& {\bf 0} & {\bf 1}
	\end{array} \right].
	$$
In addition, it can be checked that in this case $T_2$ has full row rank. As in Example \ref{ex.1symbolc}, the lines partitioning $T_2$ correspond to the partition of the Sylvester matrix $S_2$ and show that the last two block rows of $S_2$ were ``trimmed'' to get $T_2$, since all the block rows of $S_2$ have $4$ rows and the last two block rows of $T_2$ have $3$ rows and $1$ row respectively. Moreover, the rows in the second block column of $T_2$ that are indicated in bold face are the leading row-wise coefficient matrix $M_{\underline{d}}$ of $M(\la)$, which in this example coincides with the highest-row-degree coefficient matrix $M_{hr}$ of $M(\la)$. The permuted matrix $P_2 T_2$ in Lemma \ref{lemm.nesttrim} is given in this example by
   	$$ P_2 T_2 = \left[ \begin{array}{c|c} T_1 & X_2 \\ \hline 0 & M_{\underline{d}}	\end{array} \right]  = \left[ \begin{array}{ccccccc|ccccccc}
    1 & 0 & & & & & \\ & & -1 & 0 & 0 & & \\
    & & 0 & -1 & 0 & & \\ & & & & & -1 & 0 \\ \cdashline{1-7}
    && 0 & 1 & 0 &&& & & -1 & 0 & 0 & & \\
    && 0 & 0 & 1 &&& & & 0 & -1 & 0 & & \\
    &&&&& 0 & 0 & & & & & & -1 & 0 \\ \cdashline{1-7}
       	&&&&& 0 & 1 &  & & & & &  0 & 0  \\ \hline
   	&  &&&&&& {\bf 1} & {\bf 0} & & & & & \\
   	&&&&&&&  & & {\bf 0} & {\bf 1} & {\bf 0}  & & \\
   	&&&&&&&  & & {\bf 0} & {\bf 0} & {\bf 1} & & \\
   	&&&&&&& &&&&& {\bf 0} & {\bf 1}
   	\end{array} \right],
   	$$
where the dashed lines partitioning $T_1$ allow us to see that the last two block rows of $S_1$ were ``trimmed'' for getting $T_1$.  Notice also that the rows of $T_1$ are exactly the $\sum_{i=1}^m (d_i+1)$ constant row coefficients $R_{i,j}$ of the polynomial matrix $M(\la)\in\FF[\la]^{m \times (m+n)}_{\underline{d}}$ appearing in \eqref{eq.Fdbar}. It is interesting to emphasize that for this particular $M(\la) \in  \FF[\la]^{4\times 7}_{\underline{d}}$ the $7$th row of $T_1$ is zero because it corresponds to the vector coefficient of $\la$ of the $4$th row of $M(\la)$, which is zero in this case. This zero row is not trimmed because for those matrices in $\FF[\la]^{4\times 7}_{\underline{d}}$ with such vector coefficient different from zero the $7$th row of $T_1$ is not zero. This remark is related to the discussion in the paragraph just below Lemma \ref{lemm.zerorows}.
\end{example}

The last result in this section establishes a number of basic properties about the ranks and right nullities of trimmed Sylvester matrices and Sylvester matrices and is partly based on Lemma \ref{lemm.nesttrim}.

\begin{lemma} \label{lemm.ranknulltk} Let $M(\la) \in \FF[\la]^{m \times (m+n)}_{\underline{d}} \subseteq \FF[\la]^{m \times (m+n)}_{d}$, where $\underline{d} = (d_1, d_2, \ldots , d_m)$ and $d = \max_{1 \leq i \leq m} d_i$, let $S_k$ be the {\rm $k$th} Sylvester matrix of $M (\la)$, and let $T_k$ be the {\rm $k$th} trimmed Sylvester matrix of $M(\la)$, for $k = 1, 2, \ldots$. Then, the following statements hold.
\begin{itemize}
  \item[\rm (1)] $\rank (S_k) = \rank (T_k)$.
  \item[\rm (2)] $\mbox{\rm right-nullity} (S_k) =  \mbox{\rm right-nullity} (T_k)$, where the right nullity of a matrix is the dimension of its right null space.
  \item[\rm (3)] If $S_k$ has full column rank for some $k >1$, then $S_\ell$ has full column rank for all $1 \leq \ell < k$.
  \item[\rm (4)] If $T_k$ has full column rank for some $k >1$, then $T_\ell$ has full column rank for all $1 \leq \ell < k$.
  \item[\rm (5)] If $d > d_j$ for some $j$, then $S_k$ has not full row rank.
  \item[\rm (6)] If $T_k$ has full row rank, then $T_\ell$ has full row rank for all $k < \ell$.
  \item[\rm (7)] If $T_k$ has full row rank for some $k$, then $\rank (M_{\underline{d}}) = m$, $M_{\underline{d}} = M_{hr}$, and $d_i = \deg (\mbox{\rm row}_i (M(\la))$, for $i=1,\ldots, m$, where $M_{\underline{d}}$ and $M_{hr}$ are the matrices introduced in Definitions \ref{def.Mbard} and \ref{colred}, respectively.
\end{itemize}
\end{lemma}

\begin{proof}
Recall that $T_k$ is obtained from $S_k$ by removing zero rows, an operation that does not change the rank and the right nullity. This proves parts (1) and (2). Part (3) is \cite[Lemma 3.1]{VD} and holds because each matrix $S_\ell$, with $1 \leq \ell < k$, properly padded with zeros forms the first $\ell$ block columns of $S_k$. Part (4) follows from parts (1) and (3), although can also be obtained from the second equality in Lemma \ref{lemm.nesttrim}, with $k+1$ replaced by $k$, through an induction argument since $\rank (P_k T_k) = \rank (T_k)$. Part (5) holds because if $d > d_j$, then the matrix $M_d$ appearing in the definition of $S_k$ in \eqref{eq:Sylvester} has its $j$th row equal to zero and, then, the $j$th row of the last block row of $S_k$ is zero.

We only prove part (6) for $\ell = k+1$, since then the result for larger values of $\ell$ follows easily by induction. The proof requires some minor changes with respect to that of \cite[Lemma 3.2]{VD}, which states the same property for Sylvester matrices. The proof of (6) for $\ell = k+1$ is as follows: use first the equalities in Lemma \ref{lemm.nesttrim} (the second one with $k+1$ replaced by $k$, if $k >1$) and $\rank (P_k T_k) = \rank (T_k)$ to see that if $T_k$ has full row rank, then $M_{\underline{d}}$ has also full row rank; then, in a second step, use the second equality in Lemma \ref{lemm.nesttrim}, the fact that $\rank (P_{k+1} T_{k+1}) = \rank (T_{k+1})$, and the facts that $T_k$ and $M_{\underline{d}}$ have both full row rank to get that $T_{k+1}$ has also full row rank.

Finally, to prove part (7) note that, as above, if $T_k$ has full row rank, then $M_{\underline{d}}$ has also full row rank. The application of Lemma \ref{lemm.coeffmatrix} completes the proof.
\end{proof}

We emphasize that part (5) in Lemma \ref{lemm.ranknulltk} is the reason why the Sylvester matrices $S_k$ are not an adequate tool to study generic properties in $\FF[\la]^{m \times (m+n)}_{\underline{d}}$.

\section{Minimal bases in $\FF[\la]^{m \times (m+n)}_{\underline{d}}$}
\label{Sec:minbases-from-Trimmed}
This section characterizes the minimal bases in the linear space $\FF[\la]^{m \times (m+n)}_{\underline{d}}$ of $m \times (m+n)$ polynomial matrices whose row degrees are smaller than or equal to the elements of the list $\underline{d} = (d_1, d_2, \ldots , d_m)$ in terms of their trimmed Sylvester matrices. Corollary \ref{cor2:Mdfull} is the main result in this section, while the remaining results are just some of the results in \cite[Section 3]{VD} written in terms of the new trimmed Sylvester matrices instead of the Sylvester matrices. Thus, they follow immediately from Lemma \ref{lemm.ranknulltk}-(1)-(2) and the corresponding results in \cite{VD}, and their proofs are omitted. We include these results since they are used in this paper and also to emphasize that the concept of trimmed Sylvester matrices is more natural in this context since they have less rows than the corresponding Sylvester matrices, and, in fact, they can have much less rows if the elements of $\underline{d} = (d_1, d_2, \ldots , d_m)$ are highly unbalanced (think, for instance, in $m=4$, $\underline{d} = (1, 1, 1,10^4)$, $d = 10^4$).

Theorem \ref{thm.summaryprevpaper} follows from Theorem 3.3 and Corollaries 3.4 and 3.5 in \cite{VD}. Recall also that  Theorem 3.3 and Corollary 3.4 in \cite{VD} were originally proved in \cite{BKAK}.

\begin{theorem} \label{thm.summaryprevpaper} Let $M(\la) \in \FF[\la]^{m \times (m+n)}_{\underline{d}}$ be a polynomial matrix of full row rank, where $\underline{d} = (d_1, d_2, \ldots , d_m)$, let $T_k$ be the {\rm $k$th} trimmed Sylvester matrix of $M (\la)$, let $r_k$ and $n_k$ be the rank and the right nullity of $T_k$, respectively, and let $\alpha_k$ be the number of right minimal indices of $M(\la)$ equal to $k$, for $k = 1, 2, \ldots$. Then, the following statements hold.
\begin{itemize}
\item[\rm (a)] $\alpha_0=m+n-r_1=n_1$ and
\begin{equation} \label{recur}
\alpha_k = (n_{k+1}-n_k)-  (n_k-n_{k-1}) =  (r_k-r_{k-1})-(r_{k+1}-r_k), \quad k=1,2, \ldots ,
\end{equation}
where $r_0$ and $n_0$ are defined as $r_0=n_0=0$.

\item[\rm (b)] If $d'$ is the smallest index $k$ for which
\begin{equation}\label{eq:max}
    n_{k+1} - n_{k} = n, \quad \mathrm{or \;\; equivalently} \quad r_{k+1} - r_{k} = m,
 \end{equation}
 then $d'$ is the maximum right minimal index of $M(\la)$ or, equivalently, the maximum column degree of any minimal basis of the rational right null space of $M(\la)$. Moreover, for all $k$ larger than $d'$, the equalities \eqref{eq:max} still hold.

\item[\rm (c)] If $d'$ is defined as in {\rm (b)} and $\varepsilon_1 , \ldots , \varepsilon_n$ are the right minimal indices of $M(\la)$, then
\begin{equation}\label{dsum}
      \sum_{j=1}^n \varepsilon_j = \sum_{k=1}^{d'} k\alpha_k = \sum_{k=1}^{d'} k(n_{k-1}-2 n_k + n_{k+1})  = n \cdot d' - n_{d'}
= r_{d'} - m \cdot d'.
 \end{equation}
\end{itemize}
\end{theorem}

We can summarize all the relations in Theorem \ref{thm.summaryprevpaper} as follows:
$$  n_0=0, \quad r_k+n_k= k(m+n), \; 0 \le k
$$
and
\begin{equation}\label{constraint}  \left[ \begin{array}{c} \alpha_0 \\ \alpha_1 \\\alpha_2 \\ \vdots \\ \alpha_{d'}  \end{array} \right]   =
 \left[ \begin{array}{cccccc} 1 \\ -2 & 1 \\ 1 & -2 & 1 \\ & \ddots & \ddots & \ddots \\  & & 1 & -2 & 1 \end{array} \right]
  \left[ \begin{array}{c} n_1 \\ n_2 \\ \vdots \\ n_{d'} \\  n_{d'+1} \end{array} \right], \quad \sum_{k=1}^{d'} k\alpha_k =  \sum_{j=1}^n \varepsilon_j,
 \quad  \sum_{k=0}^{d'} \alpha_k =n,
\end{equation}
where the numbers $\alpha_i$'s must be nonnegative integers. Observe that the matrix in the first equality in \eqref{constraint} is invertible, which reveals that the values $\alpha_0, \alpha_1, \ldots, \alpha_{d'}$ and $\alpha_k=0$ for $k > d'$ determine uniquely the sequence $n_1, n_2, \ldots $, since $n_k = 2 n_{k-1} - n_{k-2}$ for $k > d' +1$ from \eqref{recur}, and, so, also determine uniquely the sequence $r_1, r_2, \ldots$ through the constraint $r_k + n_k = k (m+n)$.

The next theorem is \cite[Theorem 3.7]{VD} expressed in terms of trimmed Sylvester matrices for matrices in $\FF[\la]^{m\times (m+n)}_{\underline{d}}$. It provides a first characterization of minimal bases in terms of trimmed Sylvester matrices.

\begin{theorem} \label{th:finitenumbranks}
Let $M(\la) \in \FF[\la]^{m\times (m+n)}_{\underline{d}}$, where $\underline{d} = (d_1, d_2, \ldots , d_m)$, let $\widetilde{d}_i \leq d_i, i=1,\ldots, m$, be the row degrees of $M(\la)$, and let $M_{hr}$ be its highest-row-degree coefficient matrix. Let $T_k$ be the trimmed Sylvester matrices of $M(\la)$ for $k=1,2,\ldots$, and let $r_k$ and $n_k$ be the rank and the right nullity of $T_k$, respectively. Let $d'$ be the smallest index $k$ for which $n_{k+1}=n_k+n$, or equivalently, $r_{k+1}=r_k+m$. Then $M(\la)$ is a minimal basis if and only if the following rank conditions are satisfied
 \begin{equation}\label{eq:minimal}
    \rank ( M_{hr} ) = m \quad \mbox{and} \quad  r_{d'} - m \cdot d' = \sum_{i=1}^m \widetilde{d}_i  \,.
 \end{equation}
 \end{theorem}

As a consequence of Theorem \ref{th:finitenumbranks}, we prove in  Corollary \ref{cor2:Mdfull} the main result of this section that characterizes the minimal bases in $\FF[\la]^{m\times (m+n)}_{\underline{d}}$ with full row rank leading row-wise coefficient matrices $M_{\underline{d}}$ as those polynomial matrices of $\FF[\la]^{m\times (m+n)}_{\underline{d}}$ with a full row rank trimmed Sylvester matrix. Since the condition $\rank (M_{\underline{d}})= m$ is clearly generic in $\FF[\la]^{m\times (m+n)}_{\underline{d}}$, Corollary \ref{cor2:Mdfull} is key in the next sections and, therefore, we include its proof. The proof is similar to that of \cite[Corollary 3.9]{VD}, although there are also some differences related to the use of $M_{\underline{d}}$ instead of $M_d$.

\begin{corollary} \label{cor2:Mdfull}
Let $M(\la) \in \FF[\la]^{m \times (m+n)}_{\underline{d}}$,  where $\underline{d} = (d_1, d_2, \ldots , d_m)$, let $M_{\underline{d}}$ be its leading row-wise coefficient matrix introduced in Definition \ref{def.Mbard}, and let $T_k$ be the trimmed Sylvester matrices of $M(\la)$ for $k=1,2,\ldots$. Then, $M(\la)$ is a minimal basis with $\rank (M_{\underline{d}})= m$ if and only if there exists an index $k$ such that $T_k$ has full row rank. In this case, if $d'$ is the smallest index $k$ for which $T_k$ has full row rank, then $d'$ is the largest right minimal index of $M(\la)$. Moreover, all trimmed Sylvester matrices $T_k$ with $k>d'$ have then also full row rank.
\end{corollary}

\begin{proof}
If $M(\la)$ is a minimal basis with $\rank (M_{\underline{d}})= m$, then the row degrees of $M(\la)$ are precisely $d_1, d_2, \ldots , d_m$ by Lemma \ref{lemm.coeffmatrix}, and Theorem \ref{th:finitenumbranks} implies that $r_{d'} = \sum_{i=1}^m d_i + m d'$, which is the number of rows of $T_{d'}$ according to Definition \ref{def.trimSylmatrices}. Therefore $T_{d'}$ has full row rank. Then Lemma \ref{lemm.ranknulltk}-(6) implies that all $T_k$ for $k>d'$ have also full row rank.

Conversely, if there exists an index $k$ such that $T_k$ has full row rank, then  $M_{\underline{d}} = M_{hr}$, $\rank (M_{hr}) = m$, and the row degrees of $M(\la)$ are precisely $d_1,\ldots , d_m$ by Lemma \ref{lemm.ranknulltk}-(7). This also implies that $M(\la)$ has full row normal rank (recall Remark \ref{rem:NEWconvention}). Let $k_0$ be the smallest index $k$ such that $T_k$ has full row rank and denote by $r_k$ the rank of any trimmed Sylvester matrix $T_k$. Then, according to Lemma \ref{lemm.ranknulltk}-(6), $T_{k_0 + 1}$ also has full row rank and their ranks satisfy
\begin{equation} \label{eq:innercor2}
r_{k_0 + 1} - r_{k_0} = m,
\end{equation}
taking into account which are the number of rows of $T_{k_0}$ and $T_{k_0 + 1}$.
However, $r_{k_0-1} < (k_0 -1)m +\sum_{i=1}^m d_i$, because $T_{k_0-1}$ has not full row rank. Therefore, $r_{k_0} - r_{k_0 - 1} > m$ and, so, $r_{k +1} - r_{k} > m$ for all $k \leq k_0 -1$, since Theorem \ref{thm.summaryprevpaper}-(a) implies $r_{j} - r_{j-1} \geq r_{j +1} - r_{j}$ for all $j\geq 1$ because $\alpha_j \geq 0$. Therefore, $k_0$ is the smallest index $k$ such that $r_{k + 1} = r_{k} + m$, that is, $k_0 = d'$ in Theorem \ref{th:finitenumbranks} and $r_{d'} = m\;d' +\sum_{i=1}^m d_i$, since $T_{k_0} = T_{d'}$ has full row rank. Theorem \ref{th:finitenumbranks} then implies that $M(\la)$ is a minimal basis and Theorem \ref{thm.summaryprevpaper}-(b) that $k_0 = d'$ is the largest right minimal index of $M(\la)$. The fact that all $T_k$ have full row rank for $k>d'$ is again a consequence of Lemma \ref{lemm.ranknulltk}-(6).
\end{proof}

In order to illustrate Corollary \ref{cor2:Mdfull}, we revisit Example \ref{ex:1}.

\begin{example} \label{ex:1bis} Let  $m=4$, $n=3$, $\underline{d} = (0,1,1,2)$, and let $M(\la)\in  \FF[\la]^{4\times 7}_{\underline{d}}$ be the matrix in Example \ref{ex:1}. For the purpose of comparison, we consider also the following polynomial matrix $N(\la)\in  \FF[\la]^{3\times 7}$:
$$M(\la) = \left[ \begin{array}{cc|ccc|cc} 1 & 0 & & & &\\ \hline & & -1 & \la & 0 & & \\ & & 0 & -1 & \la & & \\ \hline & & & & & -1 & \la^2 \end{array} \right] ,
\quad N(\la) = \left[ \begin{array}{cc|ccc|cc} 0 & 1 & & & & \\  \hline & & \la^2 & \la & 1 \\ \hline  & & & & & \la^2 & 1 \end{array} \right] .
$$
Then, clearly $M(\la)$ and $N(\la)$ are minimal bases by Theorem \ref{minbasis_th} and $M(\la)N(\la)^T=0$. Therefore, they are dual minimal bases and the right minimal indices of $M(\la)$ are $0,2,2$. Let us deduce these properties from the results in this section. Let $T_1, T_2$, and $T_3$ be the first three trimmed Sylvester matrices of $M(\la)$. As commented in Example \ref{ex:1}, $T_2$ has full row rank, therefore Corollary \ref{cor2:Mdfull} implies that $M(\la)$ is a minimal basis. Moreover, $T_1$ has more rows than columns (see Example \ref{ex:1}) and, so, it does not have full row rank. In fact, from Example \ref{ex:1}, it is obvious that $\rank (T_1) =6$. Then, it follows from Corollary \ref{cor2:Mdfull} that $d'=2$ is the highest right minimal index of $M(\la)$, as well as that $T_3$ has full row rank. Therefore, taking into account which are the number of rows of the trimmed Sylvester matrices described in Definition \ref{def.trimSylmatrices}, the ranks $r_1, r_2, r_3$ of $T_1, T_2, T_3$ are,
$$  r_1=6, r_2=12, r_3=16,
$$
and \eqref{recur} gives $(\alpha_0,\alpha_1,\alpha_2)= (1,0,2)$ (which agrees with the row degrees of $N(\la)$).
\end{example}

 \section{Full-trimmed-Sylvester-rank polynomial matrices}
 \label{Sec:FullRank}
Section 4 in \cite{VD} characterized those matrices in $\FF[\la]^{m \times (m+n)}_d$, i.e., in the linear space of $m \times (m+n)$ polynomial matrices with degree at most $d$, whose Sylvester matrices all have full rank. Following a similar approach, in this section we characterize the polynomial matrices in $\FF[\la]^{m \times (m+n)}_{\underline{d}}$, i.e., in the linear space of $m \times (m+n)$ polynomial matrices with row degrees bounded by the entries of $\underline{d}$, whose {\em trimmed Sylvester matrices}, introduced in Definition \ref{def.trimSylmatrices}, all have full rank.  We will show that such matrices share many of the properties of the full-Sylvester-rank polynomial matrices studied in \cite[Section 4]{VD}. For instance, that they are always minimal bases, that their right minimal indices are ``almost homogeneous'', i.e., these indices differ at most by one, although their values are (very) different from those of the matrices in \cite[Section 4]{VD}.
We emphasize that this ``almost homogeneous'' property happens for any list $\underline{d}$, independently of how different are the entries of $\underline{d}$, i.e., independently of how different are the generic row degrees in $\FF[\la]^{m \times (m+n)}_{\underline{d}}$. Since the proofs in this section are very similar to those in \cite[Section 4]{VD}, most of them are omitted and only some comments on the differences are provided.

Definition \ref{def.fullsylrank} introduces the class of polynomial matrices that is studied in the rest of the paper.
\begin{definition} \label{def.fullsylrank} Let $M(\la) \in \FF[\la]^{m \times (m+n)}_{\underline{d}}$, where $\underline{d} = (d_1, d_2, \ldots, d_m)$, be a polynomial matrix with row degrees at most $d_1, d_2, \hdots , d_m$, let  $T_k$  for $k=1,2,\hdots$ be the trimmed Sylvester matrices of $M(\la)$, and let $r_k$ be their ranks. The polynomial matrix $M(\la)$ is said to have full-trimmed-Sylvester-rank if all the matrices $T_k$ have full rank, i.e., if $r_k = \min\{(km+\sum_{i=1}^m d_i) \, , \, k(m+n)\}$ for $k=1,2,\ldots$.
\end{definition}

The rank properties described in Lemma \ref{lemm.ranknulltk}-(4)-(6) imply that it is necessary and sufficient to check at most two ranks for determining whether a polynomial matrix has full-trimmed-Sylvester-rank or not. This is stated in Lemma \ref{lemm.two-ranks}. The proof of this lemma is very similar to that of Lemma 4.2 in \cite{VD} and, therefore, is omitted. The only difference is that the numbers of rows of the trimmed Sylvester matrices described in Definition \ref{def.trimSylmatrices} are different from those of the Sylvester matrices appearing in \cite[Lemma 4.2]{VD}, which results in a different value for the integer $k'$ in \eqref{eq.ktprime} since it depends on what could be called the {\em maximum total degree} (i.e., the maximum value of the sum of the row degrees) which was $dm$ in \cite{VD} for matrices in $\FF[\la]^{m \times (m+n)}_d$ and is now replaced by $\sum_{i=1}^m d_i$ for matrices in $\FF[\la]^{m \times (m+n)}_{\underline{d}}$.
Note that in Lemma \ref{lemm.two-ranks}, as well as in the rest of this paper, the {\em ceiling} function of a real number $x$ is often used and is denoted by $\lceil x \rceil$. Recall that $\lceil x \rceil$ is the smallest integer that is larger than or equal to $x$.

\begin{lemma} \label{lemm.two-ranks}Let $M(\la) \in \FF[\la]^{m \times (m+n)}_{\underline{d}}$, where $\underline{d} = (d_1, d_2, \ldots, d_m)$, let $T_k$ for $k=1,2,\ldots$ be the trimmed Sylvester matrices of $M(\la)$, and let
\begin{equation} \label{eq.ktprime}
k' := \left\lceil \frac{\sum_{i=1}^m d_i}{n} \right\rceil \quad \mbox{and} \quad n k' = \sum_{i=1}^m d_i + t, \quad \mbox{where $0\leq t < n$.}
\end{equation}
Then the following statements hold.
\begin{itemize}
\item[\rm (a)] $k'$ is the smallest index $k$ for which the number of columns of $T_k$ is larger than or equal to the number of rows of $T_k$.
\item[\rm (b)] If $k' > 1$ and $t>0$, then $M(\la)$ has full-trimmed-Sylvester-rank if and only if $T_{k'-1}$ has full column rank and $T_{k'}$ has full row rank.
\item[\rm (c)] If $k' = 1$ or $t=0$, then $M(\la)$ has full-trimmed-Sylvester-rank if and only if $T_{k'}$ has full row rank.
\end{itemize}
\end{lemma}

Once full-trimmed-Sylvester-rank polynomial matrices have been characterized, we establish some of their properties. First, according to Corollary \ref{cor2:Mdfull}, polynomial matrices with full-trimmed-Sylvester-rank are minimal bases whose leading row-wise coefficient matrices $M_{\underline{d}}$ have full rank and with row degrees exactly equal to $d_1, d_2, \ldots ,d_m$, as a consequence of Lemma \ref{lemm.coeffmatrix}. This is stated in Theorem \ref{thm.propsfullSylvrank}-(b). In addition, Lemma \ref{lemm.two-ranks} implies that $k'$ in \eqref{eq.ktprime} is the smallest index $k$ for which $T_k$ has full row rank for any full-trimmed-Sylvester-rank matrix. Combining this fact with Corollary \ref{cor2:Mdfull}, we obtain that $k'$ is the largest right minimal index of any full-trimmed-Sylvester-rank matrix, or, equivalently, the degree of any of its dual minimal bases. This is stated in Theorem \ref{thm.propsfullSylvrank}-(c). Finally, by definition, $M(\la)$ has full-trimmed-Sylvester rank if and only if the ranks $r_k$ of their trimmed Sylvester matrices $T_k$ are given by $r_k = \min\{(km+\sum_{i=1}^m d_i) \, , \, k(m+n)\}$. Then, the right minimal indices of any full-trimmed-Sylvester-rank polynomial matrix are fixed by the recurrence in \eqref{recur} and they are given by \eqref{eq.statfullSylrank}, as can be checked through some algebraic manipulations (the reader can find in the proof of \cite[Theorem 4.3]{VD} the details of similar manipulations). Conversely, as remarked in the paragraph just below Theorem \ref{thm.summaryprevpaper}, the right minimal indices in \eqref{eq.statfullSylrank} also determine that the ranks of the trimmed Sylvester matrices are $r_k = \min\{(km+\sum_{i=1}^m d_i) \, , \, k(m+n)\}$. This discussion leads to the characterization of full-trimmed-Sylvester-rank matrices in terms of their right minimal indices given in Theorem \ref{thm.propsfullSylvrank}-(a).

\begin{theorem} \label{thm.propsfullSylvrank} Let $M(\la) \in \FF[\la]^{m \times (m+n)}_{\underline{d}}$, where $\underline{d} = (d_1, d_2 , \ldots, d_m)$, let $\alpha_k$ be the number of right minimal indices of $M(\la)$ equal to $k$, let $k'$ and $t$ be defined as in \eqref{eq.ktprime}, and let $M_{\underline{d}}$ be the leading row-wise coefficient matrix of $M(\la)$ introduced in Definition \ref{def.Mbard}. Then the following statements hold.
\begin{itemize}
\item[\rm (a)]
$M(\la)$ has full-trimmed-Sylvester-rank if and only if the right minimal indices of $M(\la)$ are
\begin{equation} \label{eq.statfullSylrank}
\alpha_{k'-1} = t, \quad \alpha_{k'} = n-t, \quad \mbox{and} \quad  \alpha_j = 0 \; \; \mbox{for $j\notin \{k'-1,k'\}$}.
\end{equation}
\item[\rm (b)] If $M(\la)$ has full-trimmed-Sylvester-rank, then $M(\la)$ is a minimal basis with $\rank (M_{\underline{d}}) = m$, and with row degrees exactly equal to $d_1, d_2, \ldots, d_m$.
\item[\rm (c)] If $M(\la)$ has full-trimmed-Sylvester-rank, then the degree of any minimal basis dual to $M(\la)$ is $k'$.
\end{itemize}

\end{theorem}

If Theorem \ref{thm.propsfullSylvrank} is compared with \cite[Theorem 4.3]{VD}, then we observe that parts-(a) are different in both results since Theorem 4.3-(a) in \cite{VD} states in addition that the complete eigenstructure of full-Sylvester-rank matrices consists only of right minimal indices. In contrast, full-trimmed-Sylvester-rank matrices in $\FF[\la]^{m \times (m+n)}_{\underline{d}}$ have also infinite elementary divisors, except in the case that all the entries in $\underline{d}$ are equal. This is proved in the next theorem, where the infinite elementary divisors of $M(\la)$ are the elementary divisors associated to the zero eigenvalue of $\mbox{rev}_d \, M(\la) = \la^d \, M(1/\la)$.

\begin{theorem} \label{thm.completeeigen} Let $M(\la) \in \FF[\la]^{m \times (m+n)}_{\underline{d}}$, where $\underline{d} = (d_1, d_2 , \ldots, d_m)$, let $\alpha_k$ be the number of right minimal indices of $M(\la)$ equal to $k$, let $k'$ and $t$ be defined as in \eqref{eq.ktprime}, and let $d = \max_{1 \leq i \leq m} d_i$. Then, $M(\la)$ has full-trimmed-Sylvester-rank if and only if the complete eigenstructure of $M(\la)$ consists of
\begin{itemize}
\item[\rm (1)] the right minimal indices described by
\[
\alpha_{k'-1} = t, \quad \alpha_{k'} = n-t, \quad \mbox{and} \quad  \alpha_j = 0 \; \; \mbox{for $j\notin \{k'-1,k'\}$},
\]
\item[\rm (2)] and one infinite elementary divisor of degree $d-d_i$ for each $d_i < d$, $i=1,2, \ldots , m$.
\end{itemize}
\end{theorem}

\begin{proof}
It is clear that if (1) and (2) hold, then $M(\la)$ has full-trimmed-Sylvester-rank as a consequence of Theorem \ref{thm.propsfullSylvrank}-(a).

Conversely, if $M(\la)$ has full-trimmed-Sylvester-rank, then its right minimal indices are the ones described in (1) by Theorem \ref{thm.propsfullSylvrank}-(a). Moreover, $M(\la)$ is a minimal basis with row degrees exactly equal to $d_1, d_2, \ldots , d_m$ by Theorem \ref{thm.propsfullSylvrank}-(b). Therefore, $M(\la)$ has no left minimal indices, since it has full row (normal) rank, and $M(\la)$ has no finite elementary divisors by Theorem \ref{minbasis_th}. The only remaining part of the complete eigenstructure of $M(\la)$ to be determined are the infinite elementary divisors. For this purpose, let $R_i (\la)$ be the rows of $M(\la)$, for $i=1,\ldots, m$, and note that
\[
\mbox{rev}_d \, M(\la) = \left[ \begin{array}{ccc}
                        \la^{d - d_1} &  & \\
                         & \ddots &  \\
                         &  & \la^{d - d_m}
                      \end{array} \right]
                       \left[ \begin{array}{c}
                         \mbox{rev}_{d_1} \, R_1 (\la) \\
                         \vdots \\
                         \mbox{rev}_{d_m} \, R_m (\la)
                      \end{array} \right] =: D(\la) H(\la).
\]
The matrix $H(\la)$ whose rows are the reversals of the rows of $M(\la)$ is also a minimal basis by \cite[Theorem 3.2]{DDM2009} and, therefore, has no finite elementary divisors. This means that it can be extended to a unimodular matrix by adding rows  \cite[Chapter 6]{Kai80} and, so, $[D(\la) \; 0_{m \times n}]$ is, modulo a permutation of its diagonal entries, the Smith form of $\mbox{rev}_d \, M(\la)$ and those $\la^{d - d_i}$ such that $d - d_i >0$ are the infinite elementary divisors of $M(\la)$.
\end{proof}

\section{Genericity of full-trimmed-Sylvester-rank matrices in $\FF[\la]^{m \times (m+n)}_{\underline{d}}$} \label{Sec:Genericity}
In this section we extend the genericity results of \cite[Section 5]{VD} from the linear space $\FF[\la]^{m \times (m+n)}_d$ to the linear space $\FF[\la]^{m \times (m+n)}_{\underline{d}}$. The results are stated both in terms of algebraic sets and of open and dense sets with respect to a standard Euclidean metric in $\FF[\la]^{m \times (m+n)}_{\underline{d}}$. This second view was not developed in \cite{VD}, and is included here because it completes the genericity results in a natural way and because it will be applied in the future for describing the sets of polynomial matrices with bounded rank and degree from a different perspective than the one presented in \cite{dmy-dop-2017}.  The first part of this section introduces some general concepts and basic results that are needed to state the main results.

As explained in Section \ref{Sec:Preliminaries}, $\FF[\la]^{m \times (m+n)}_{\underline{d}}$, where $\underline{d} = (d_1, d_2 , \ldots , d_m)$, is a linear space of dimension $(m+n) \sum_{i=1}^m(d_i+1)$ over the field $\FF$, which we restrict here and in the rest of this paper to  $\RR$ or $\CC$. We identify $\FF[\la]^{m \times (m+n)}_{\underline{d}}$ with $\FF^{(m+n) \sum_{i=1}^m(d_i+1)}$. Such identification can be made, for instance, by mapping each polynomial matrix $M(\la) \in \FF[\la]^{m \times (m+n)}_{\underline{d}}$ with rows $R_i(\la) = R_{i,0} + R_{i,1} \la + \cdots + R_{i,d_i} \la^{d_i}$, $i=1,2, \ldots m$, into a long row vector $V(M):= [R_{1,0} \; R_{1,1} \; \cdots \; R_{1,d_1}\; \cdots \; R_{m,0} \; R_{m,1} \; \cdots \; R_{m,d_m} ] \in \FF^{(m+n) \sum_{i=1}^m (d_i+1)}$. The mapping $M(\la) \mapsto V(M)$ is clearly a linear bijection between $\FF[\la]^{m \times (m+n)}_{\underline{d}}$ and $\FF^{(m+n) \sum_{i=1}^m (d_i+1)}$. In addition, this linear mapping $V$ is a bijective isometry if $\FF^{(m+n) \sum_{i=1}^m (d_i+1)}$ is endowed with the standard Euclidean distance and $\FF[\la]^{m \times (m+n)}_{\underline{d}}$ with the following distance: for any $M(\la), \widetilde{M} (\la) \in \FF[\la]^{m \times (m+n)}_{\underline{d}}$ with rows given, respectively, by $R_i(\la) = R_{i,0} + R_{i,1} \la + \cdots + R_{i,d_i} \la^{d_i}$ and $\widetilde{R}_i(\la) = \widetilde{R}_{i,0} + \widetilde{R}_{i,1} \la + \cdots + \widetilde{R}_{i,d_i} \la^{d_i}$ for $i=1,2,\ldots, m$, the distance between $M(\la)$ and  $\widetilde{M} (\la)$ is defined as
\begin{equation}\label{eq.poldistance}
\rho (M,\widetilde{M}) := \sqrt{\sum_{i=1}^{m} \sum_{j=0}^{d_i} \|R_{i,j} - \widetilde{R}_{i,j} \|_2^2} \; ,
\end{equation}
where $\| \cdot \|_2$ is the standard Euclidean vector norm. It is obvious that $\rho (M,\widetilde{M})$ is equal to the standard Euclidean distance between the vectors $V(M)$ and $V(\widetilde{M})$ in $\FF^{(m+n) \sum_{i=1}^m (d_i+1)}$, i.e., $\rho (M,\widetilde{M}) = \|V(M) -  V(\widetilde{M}) \|_2$. The distance \eqref{eq.poldistance} can be expressed more compactly in terms of the matrix coefficients of $M(\la)$ and $\widetilde{M} (\la)$, as well as in terms of their first Sylvester and first trimmed Sylvester matrices, because if $d = \max_{1 \leq i \leq m} d_i$, $M(\la) = M_0 + M_1  \la + \cdots + M_d \la^d$, and $\widetilde{M}(\la) = \widetilde{M}_0 + \widetilde{M}_1  \la + \cdots + \widetilde{M}_d \la^d$, then
\begin{equation}\label{eq.poldistance2}
\rho (M,\widetilde{M}) = \sqrt{\sum_{i=0}^{d} \|M_i - \widetilde{M}_{i} \|_F^2} = \|S_1 (M) - S_1 (\widetilde{M})\|_F = \|T_1 (M) - T_1 (\widetilde{M})\|_F,
\end{equation}
where $\|\cdot \|_F$ is the matrix Frobenius norm \cite{stewart-sun} and the last equality holds because trimmed Sylvester matrices are obtained from Sylvester matrices by removing rows that are zero for every polynomial matrix in $\FF[\la]^{m \times (m+n)}_{\underline{d}}$. The distance \eqref{eq.poldistance} allows us to consider $\FF[\la]^{m \times (m+n)}_{\underline{d}}$ as a metric space, and, so, to define in it open and closed sets, as well as closures and any other topological concept. Moreover, the linear bijective isometry $M(\la) \mapsto V(M)$ allows us to see that such concepts are fully equivalent to the corresponding ones in the standard Euclidean metric in $\FF[\la]^{m \times (m+n)}_{\underline{d}}$.

Sometimes in this section we need to consider $\CC[\la]^{m \times (m+n)}_{\underline{d}}$ as a linear space over $\RR$ and to identify $\CC[\la]^{m \times (m+n)}_{\underline{d}}$ with $\RR^{2 (m+n) \sum_{i=1}^m(d_i+1)}$ through the linear bijection $M(\la) \mapsto W(M) := [\mbox{Re}(V(M)) \; \; \mbox{Im}(V(M))]$, where $V(M)$ is the row vector defined above for any $M(\la) \in \CC[\la]^{m \times (m+n)}_{\underline{d}}$ and $\mbox{Re}(V(M))$ and $\mbox{Im}(V(M))$ are its entry-wise real and imaginary vector parts. In addition,
note that for any $M(\la), \widetilde{M} (\la) \in \CC[\la]^{m \times (m+n)}_{\underline{d}}$, $\rho (M,\widetilde{M}) = \|V(M) -  V(\widetilde{M}) \|_2 = \|W(M) -  W(\widetilde{M}) \|_2$. Therefore, $W$ is a linear bijective isometry between
$\CC[\la]^{m \times (m+n)}_{\underline{d}}$ endowed with the distance \eqref{eq.poldistance}
and $\RR^{2 (m+n) \sum_{i=1}^m(d_i+1)}$ endowed with the standard Euclidean distance, which allows us to identify open, closed sets, and any other topological concepts, in these two metric spaces.

Next, we recall that an {\em algebraic set} in $\FF^p$ (here $\FF = \RR$ or $\FF = \CC$) is the set of common zeros of a finite number of multivariable polynomials with $p$ variables and coefficients in $\FF$, and that an algebraic set is {\em proper} if it is not the whole set $\FF^p$. With these concepts at hand, the standard definition of genericity of Algebraic Geometry is as follows: a {\em generic set of $\FF^p$ is a subset of $\FF^p$ whose complement is contained in a proper algebraic set}. This definition extends to the corresponding one of {\em generic set of $\RR[\la]^{m \times (m+n)}_{\underline{d}}$}
through the bijection $M(\la) \mapsto V(M)$, and to the corresponding one of
{\em generic set of $\CC[\la]^{m \times (m+n)}_{\underline{d}}$} through the bijection $M(\la) \mapsto W(M)$. Note that one can also define generic sets of $\CC[\la]^{m \times (m+n)}_{\underline{d}}$ through the bijection
$M(\la) \mapsto V(M)$, but this is not the definition that we need in this paper, although it may be of interest in other contexts, since for algebraic genericity purposes, we need to consider $\CC[\la]^{m \times (m+n)}_{\underline{d}}$ as a linear space over $\RR$.

Generic sets in $\FF^p$ satisfy the important property stated in Theorem \ref{thm.generopen}, which is well-known but that we have not found explicitly stated anywhere. Therefore, we include a proof that relies only on the following two very basic results of Euclidean Topology in $\FF^p$: a closed set is a set whose complement is an open set, and vice versa, and a set is closed if and only if it contains all of its limit points. It is interesting to observe that another definition of ``generic set'' in $\FF^p$, which is also often used, is that a set is generic if it contains a subset that is open and dense in $\FF^p$ with respect to the standard Euclidean topology. Therefore, Theorem \ref{thm.generopen} proves that the generic sets in the sense of Algebraic Geometry are also generic in this topological sense. However, the reverse implication is not true.

\begin{theorem} \label{thm.generopen} Let $\FF = \RR$ or $\FF = \CC$ and let $\mathcal{A}$ be a generic set of $\FF^p$. Then, there exists a subset $\mathcal{B} \subseteq \mathcal{A}$ such that $\mathcal{B}$ is open and dense in $\FF^p$ with respect to the standard Euclidean topology. Moreover, if the complement of $\mathcal{A}$ in $\FF^p$ is a proper algebraic set, then $\mathcal{A}$ itself is open and dense in $\FF^p$.
\end{theorem}

\begin{proof} If $\mathcal{A} = \FF^p$, then take $\mathcal{B} = \mathcal{A}$ and the proof is finished. So, we assume throughout the rest of the proof that $\mathcal{A} \ne \FF^p$. Let $\mathcal{A}^c$ ($\ne \emptyset$) be the complement of $\mathcal{A}$ in $\FF^p$. Then, by definition of genericity, there exists a proper algebraic set $\mathcal{C}$ of $\FF^p$ such that $\mathcal{A}^c \subseteq \mathcal{C} \ne \emptyset$, or, equivalently, $\mathcal{C}^c \subseteq \mathcal{A} \ne \FF^p$. The goal is to prove that $\mathcal{C}^c$ is open and dense in $\FF^p$. Therefore, $\mathcal{C}^c$ can be taken as the set $\mathcal{B}$ mentioned in the statement. By definition, there exist some multivariable polynomials $p_1 (x_1, \ldots , x_p), \ldots , p_q (x_1, \ldots , x_p)$ such that
\[
\mathcal{C} = \{
x
\in \FF^p \, : \, p_i (x) = 0 \;\mbox{for $i=1, \ldots , q$}
\},
\]
where at least one of the polynomials $p_i$ is not identically zero because $\mathcal{C}$ is proper. Next, we prove that $\mathcal{C}$ is closed, i.e., that $\mathcal{C}^c$ is open. For this purpose, let $\{y^{(k)}\}_{k=0}^\infty \subseteq \mathcal{C}$ be a sequence such that $\lim_{k \rightarrow \infty} y^{(k)} = y$. Then, since multivariable polynomials are continuous functions and $p_i(y^{(k)}) = 0$ for all $k$ and for $i=1, \ldots , q$,
$$0 = \lim_{k \rightarrow \infty}  p_i(y^{(k)}) = p_i(\lim_{k \rightarrow \infty} y^{(k)}) = p_i(y),$$
which proves that $y \in \mathcal{C}$ and, so, that $\mathcal{C}$ is closed. Next, we prove that $\mathcal{C}^c$ is dense in $\FF^p$. Let $z \notin \mathcal{C}^c$. Then, $z \in \mathcal{C}$. Let $p_j (x)$ be one of the polynomials defining $\mathcal{C}$ that is not identically zero and let $x_\ell$ be a variable appearing in such polynomial. Then, $p_j (z) =0$. Define the univariate polynomial $q(s) =
p_j (z_1, \ldots, z_{\ell -1}, s , z_{\ell +1}, \ldots , z_p)$, where $s \in \FF$, which satisfies $q(z_\ell) = 0$. Since the number of roots of $q(s)$ is finite we can construct a sequence $\{s^{(k)}\}_{k=0}^\infty \subseteq \FF$ such that $q(s^{(k)}) \ne 0$ and $\lim_{k \rightarrow \infty} s^{(k)} = z_\ell$. Finally, we define the sequence $\{z^{(k)} = (z_1, \ldots, z_{\ell-1}, s^{(k)}  , z_{\ell+1}, \ldots , z_p)\}_{k=0}^\infty \subseteq \FF^p$, which satisfies $p_j (z^{(k)}) \ne 0$ and $\lim_{k \rightarrow \infty} z^{(k)} = z$. Thus, $\{ z^{(k)} \}_{k=0}^\infty \subset \mathcal{C}^c$ and $z$ is a limit point of $\mathcal{C}^c$. The ``moreover part'' of the theorem follows by taking in that case $\mathcal{A}^c = \mathcal{C}$.
\end{proof}

The basic concepts on ``genericity'' that have been refreshed above allow us to introduce some nomenclature. In the rest of the paper, expressions as ``generically the polynomial matrices in $\FF[\la]^{m \times (m+n)}_{\underline{d}}$ have the property $\mathcal{P}$'' have the precise meaning of ``the polynomial matrices of $\FF[\la]^{m \times (m+n)}_{\underline{d}}$ that satisfy property $\mathcal{P}$ are a  generic set of $\FF[\la]^{m \times (m+n)}_{\underline{d}}$''.

We can now state the main result of this section.

\begin{theorem} \label{thm.gensylvfullrank} Let $\mathrm{TSyl}[\la]^{m \times (m+n)}_{\underline{d}}$ be the set of polynomial matrices that have full-trimmed-Sylvester-rank in $\FF[\la]^{m \times (m+n)}_{\underline{d}}$. Then, the complement of $\mathrm{TSyl}[\la]^{m \times (m+n)}_{\underline{d}}$ is a proper algebraic set of $\FF[\la]^{m \times (m+n)}_{\underline{d}}$ and, so, $\mathrm{TSyl}[\la]^{m \times (m+n)}_{\underline{d}}$ is a generic set of $\FF[\la]^{m \times (m+n)}_{\underline{d}}$. Moreover, $\mathrm{TSyl}[\la]^{m \times (m+n)}_{\underline{d}}$ is an open and dense subset of $\FF[\la]^{m \times (m+n)}_{\underline{d}}$ with respect to the Euclidean metric defined in \eqref{eq.poldistance}.
\end{theorem}

\begin{proof} The proof of that the complement of $\mathrm{TSyl}[\la]^{m \times (m+n)}_{\underline{d}}$ is a proper algebraic set of $\FF[\la]^{m \times (m+n)}_{\underline{d}}$ is very similar to that of Theorem 5.1 in \cite{VD} and only requires to replace the Sylvester matrices used in \cite{VD} by the trimmed Sylvester matrices. More precisely, assume that $k'$ and $t$ in \eqref{eq.ktprime} satisfy $k' >1$ and $t >0$, since the proof in other cases is similar. Then, the complement of $\mathrm{TSyl}[\la]^{m \times (m+n)}_{\underline{d}}$ is the set of matrices of $\FF[\la]^{m \times (m+n)}_{\underline{d}}$ that satisfy $\det(T_{k'-1}^* T_{k'-1}) \cdot \det(T_{k'} T_{k'}^*) = 0$, according to Lemma \ref{lemm.two-ranks}-(b). Next, the same arguments presented in \cite[Theorem 5.1]{VD} prove that this set is a proper algebraic set of $\FF[\la]^{m \times (m+n)}_{\underline{d}}$, where the proof of the ``properness'' relies on Theorem \ref{thm:dualbasis}. Once this is established, Theorem \ref{thm.generopen} can be combined with the linear bijective isometries $M(\la) \mapsto V(M)$, when $\FF = \RR$, or $M(\la) \mapsto W(M)$, when $\FF = \CC$, defined in the first part of this section to prove immediately that  $\mathrm{TSyl}[\la]^{m \times (m+n)}_{\underline{d}}$ is an open and dense subset of $\FF[\la]^{m \times (m+n)}_{\underline{d}}$ with respect to the metric in \eqref{eq.poldistance}.
\end{proof}

An immediate consequence of Theorem \ref{thm.gensylvfullrank} is that generically  the polynomial matrices in $\FF[\la]^{m \times (m+n)}_{\underline{d}}$ have all the properties satisfied by full-trimmed-Sylvester-rank matrices, in particular those established in Theorems \ref{thm.propsfullSylvrank} and \ref{thm.completeeigen}. Moreover, the fact, proved in Theorem \ref{thm.gensylvfullrank}, that $\mathrm{TSyl}[\la]^{m \times (m+n)}_{\underline{d}}$ is open and dense and the definition of these concepts allow us to state the following corollary, whose simple proof is omitted.

\begin{corollary} \label{cor:genproperties} Let $\underline{d} = (d_1, d_2 , \ldots , d_m)$,  $k'$ and $t$ be defined as in \eqref{eq.ktprime},  $\rho$ be the distance in $\FF[\la]^{m \times (m+n)}_{\underline{d}}$ defined in \eqref{eq.poldistance}, and $\mathrm{TSyl}[\la]^{m \times (m+n)}_{\underline{d}} \subset \FF[\la]^{m \times (m+n)}_{\underline{d}}$ be the set of full-trimmed-Sylvester-rank polynomial matrices. Then, the following statements hold.
\begin{itemize}
\item[\rm (a)] For every $M(\la) \in \FF[\la]^{m \times (m+n)}_{\underline{d}}$ and every $\epsilon >0$, there exists a polynomial matrix $\widetilde{M}(\la) \in \mathrm{TSyl}[\la]^{m \times (m+n)}_{\underline{d}}$ such that $\rho (M, \widetilde{M}) < \epsilon$, which, therefore, satisfies
    \begin{itemize}
      \item[\rm (a1)] $\widetilde{M}(\la)$ is a minimal basis with full row rank leading row-wise coefficient matrix, and with row degrees exactly equal to $d_1, d_2, \ldots, d_m$,
      \item[\rm (a2)]$\widetilde{M}(\la)$ has $n$ right minimal indices, $t$ of them equal to $k'-1$ and $n-t$ equal to $k'$.
    \end{itemize}
\item[\rm (b)] For every $\widetilde{M}(\la) \in \mathrm{TSyl}[\la]^{m \times (m+n)}_{\underline{d}}$, there exists a number $\epsilon >0$ such that every polynomial matrix $M(\la) \in \FF[\la]^{m \times (m+n)}_{\underline{d}}$ satisfying
    $\rho (M, \widetilde{M}) < \epsilon$ has full-trimmed-Sylvester-rank.
\end{itemize}
\end{corollary}

Among the generic properties of the polynomial matrices in $\FF[\la]^{m \times (m+n)}_{\underline{d}}$ perhaps the most remarkable one is the ``almost homogeneity'' of their right minimal indices (recall that this means that they differ at most by one) displayed in Corollary \ref{cor:genproperties}-(a2). The surprising fact is that this generic property holds independently on how different are the generic row degrees $d_1, d_2 , \ldots , d_m$ of the matrices in $\FF[\la]^{m \times (m+n)}_{\underline{d}}$, which can be arbitrarily ``inhomogeneous''. We emphasize that the generic values of these ``almost homogeneous right minimal indices'' are fully determined by the constraint that their sum is equal to $\sum_{i=1}^{m} d_i$. The genericity of this property was proved in \cite[Section 5]{VD} only in the space of $m \times (m+n)$ polynomial matrices with degree at most $d$, which generically have all their row degrees equal to $d$.

Finally, note that, taking into account \eqref{eq.poldistance2}, the distance $\rho (M, \widetilde{M})$ in $\FF[\la]^{m \times (m+n)}_{\underline{d}}$ is equivalent to the distance
\begin{equation} \label{eq.spectraldistance}
\rho_2 (M, \widetilde{M}) := \|T_1 (M) - T_1 (\widetilde{M})\|_2,
\end{equation}
where $\|\cdot \|_2$ is the standard matrix spectral norm or maximum singular value of the considered matrix. This equivalence follows from $\|T_1 (M) - T_1 (\widetilde{M})\|_F/ \sqrt{\min\{m + \sum_{i=1}^{m} d_i , m+n\}} \leq \|T_1 (M) - T_1 (\widetilde{M})\|_2 \leq \|T_1 (M) - T_1 (\widetilde{M})\|_F$ \cite{stewart-sun}. The distance $\|T_1 (M) - T_1 (\widetilde{M})\|_2$ will be also used in some of the results of the rest of the paper since it leads to sharper bounds.

\section{Robustness of full-trimmed-Sylvester-rank matrices in $\FF[\la]^{m \times (m+n)}_{\underline{d}}$} \label{Sec:Smoothness}

Corollary \ref{cor:genproperties}-(b) establishes that every  full-trimmed-Sylvester-rank matrix $M(\la)$ is robust under perturbations in the sense that all the polynomial matrices of $\FF[\la]^{m \times (m+n)}_{\underline{d}}$ in a neighborhood of $M(\la)$ are also full-trimmed-Sylvester-rank matrices. In this section, we estimate the size of the corresponding neighborhood of robustness. In addition, we also characterize when any minimal basis $M(\la) \in \FF[\la]^{m \times (m+n)}_{\underline{d}}$, which may have not full-trimmed-Sylvester-rank, is robust under perturbations, again in the sense that all the polynomial matrices in a neighborhood of $M(\la)$ are also minimal bases. The proofs of the results in this section are omitted, since they are essentially equal to the proofs in \cite[Section 6]{VD}, with the main differences coming from replacing Sylvester matrices and their properties by trimmed Sylvester matrices and their properties.

Recall that for any polynomial matrix $P(\la) \in \FF[\la]^{m \times (m+n)}_{\underline{d}}$, its $k$th trimmed Sylvester matrix is denoted by $T_k (P)$. In this section, we use the distance $\|T_1 (P)- T_1 (\widetilde{P})\|_2 = \|T_1 (P - \widetilde{P})\|_2$ between any two polynomial matrices $P(\la), \widetilde{P} (\la) \in \FF[\la]^{m \times (m+n)}_{\underline{d}}$, which was already introduced in \eqref{eq.spectraldistance}. The singular values of any constant matrix $A\in \FF^{p \times q}$ are denoted by $\sigma_1 (A) \geq \cdots \geq \sigma_{\min\{p,q\}} (A)$, here and in the rest of the paper.
 We first prove the next simple result that follows from \cite[Lemma 6.2]{VD}.

\begin{lemma} \label{C} Let $P(\la) \in \FF[\la]^{m \times (m+n)}_{\underline{d}}$. Then the following inequalities hold for the trimmed Sylvester matrices of $P(\la)$:
$$ \| T_1 (P)\|_2 \le \|T_{k} (P)\|_2 \le \sqrt{k} \cdot \| T_1 (P)\|_2.  $$
\end{lemma}

\begin{proof}
Lemma 6.2 in \cite{VD} proves that $\| S_1 (P)\|_2 \le \|S_{k} (P)\|_2 \le \sqrt{k} \cdot \| S_1 (P)\|_2$ for the Sylvester matrices of $P(\la)$. Then, note that $\|T_{k} (P)\|_2 = \|S_{k} (P)\|_2$ for any $k$, because $T_{k} (P)$ is obtained from $S_{k} (P)$ by removing zero rows, which does not change the largest singular value.
\end{proof}

The next theorem proves that a minimal basis with row degrees at most $d_i,i=1,\hdots,m$, is robust inside $\FF[\la]^{m \times (m+n)}_{\underline{d}}$ if and only if its row degrees are all maximal or, equivalently, according to Lemma \ref{lemm.coeffmatrix}, if and only if its leading row-wise coefficient matrix introduced in Definition \ref{def.Mbard} has full rank. Theorem \ref{thm.smoothminbases} is the counterpart in $\FF[\la]^{m \times (m+n)}_{\underline{d}}$ of \cite[Theorem 6.3]{VD} in $\FF[\la]^{m \times (m+n)}_{d}$. The omitted proof is based on Corollary \ref{cor2:Mdfull} and Lemma \ref{C}.

\begin{theorem} \label{thm.smoothminbases} Let $M(\la) \in \FF[\la]^{m \times (m+n)}_{\underline{d}}$, where $\underline{d} = (d_1,\ldots , d_m)$, be a minimal basis and let $M_{\underline{d}}$ be its leading row-wise coefficient matrix introduced in Definition \ref{def.Mbard}. Then the following statements hold:
\begin{itemize}
\item[\rm (a)] If $\rank ( M_{\underline{d}}) < m$, then for all $\epsilon >0$ there exists a polynomial matrix $\widetilde{M}(\la)\in \FF[\la]^{m \times (m+n)}_{\underline{d}}$ that is not a minimal basis and satisfies $\|T_1 (M)- T_1 (\widetilde{M})\|_2 < \epsilon$. That is, as close as we want to $M(\la)$ there are polynomial matrices that are not minimal bases.

\item[\rm (b)] If $\rank ( M_{\underline{d}} ) = m$, then there exists an index $k$ such that $T_k (M)$ has full row rank and every polynomial matrix $\widetilde{M}(\la) \in \FF[\la]^{m \times (m+n)}_{\underline{d}}$ that satisfies
    \begin{equation} \label{eq.smoothneigh1}
    \|T_1 (M)- T_1 (\widetilde{M})\|_2 < \frac{\sigma_{km+\sum_{i=1}^m d_i} (T_{k}(M))}{\sqrt{k}}
    \end{equation}
    is a minimal basis with $\rank ( \widetilde{M}_{\underline{d}} ) =m$.  That is, all the polynomial matrices sufficiently close to $M(\la)$ are minimal bases with full rank leading row-wise coefficient matrix.
\end{itemize}
\end{theorem}

We now give bounds on the size of the robustness neighborhoods for the full-trimmed-Sylvester-rank property. Theorem \ref{thm.smoothfullSylvrank} is the counterpart of \cite[Theorem 6.6]{VD} and its omitted proof is based on Lemmas \ref{lemm.two-ranks} and \ref{C}.

\begin{theorem} \label{thm.smoothfullSylvrank} Let $M(\la) \in \FF[\la]^{m \times (m+n)}_{\underline{d}}$, where $\underline{d} = (d_1,\ldots , d_m)$, be a polynomial matrix with full-trimmed-Sylvester-rank and let $k'$ and $t$ be defined as in \eqref{eq.ktprime}. Then the following statements hold:
\begin{itemize}
\item[\rm (a)] If $k' > 1$ and $t >0$, then every $\widetilde{M} (\la) \in \FF[\la]^{m \times (m+n)}_{\underline{d}}$ such that
\[
\|T_1 (M)- T_1 (\widetilde{M})\|_2 < \min \left\{
\frac{\sigma_{(k'-1)(m+n)} (T_{k' -1}(M))}{\sqrt{k'-1}} \, , \,
\frac{\sigma_{k'm+\sum_{i=1}^m d_i} (T_{k'}(M))}{\sqrt{k'}} \right\} \,
\]
has full-trimmed-Sylvester-rank.

\item[\rm (b)] If $k' = 1$ or $t = 0$,
then every $\widetilde{M} (\la) \in \FF[\la]^{m \times (m+n)}_{\underline{d}}$ such that
\[
\|T_1 (M)- T_1 (\widetilde{M})\|_2 <
\frac{\sigma_{k'm+\sum_{i=1}^m d_i} (T_{k'}(M))}{\sqrt{k'}} \,
\]
has full-trimmed-Sylvester-rank.
\end{itemize}
\end{theorem}

In the special case that $\sum_{i=1}^m d_i \leq n$, the first trimmed Sylvester matrix $T_1(M)$ is ``flat'' and Theorem \ref{thm.smoothfullSylvrank} can be slightly improved by showing that the estimation of the size of the robustness neighborhood of the full-trimmed-Sylvester-rank property is sharp, i.e., it cannot be extended. This is presented in Corollary \ref{cor.sharprobust}, which is the counterpart in $\FF[\la]^{m \times (m+n)}_{\underline{d}}$ of \cite[Corollary 6.8]{VD} and has a similar proof that is omitted.

\begin{corollary} \label{cor.sharprobust} Let $M(\la) \in \FF[\la]^{m \times (m+n)}_{\underline{d}}$, where $\underline{d} = (d_1,\ldots , d_m)$. If $\sum_{i=1}^m d_i \leq n$, then the following statements hold:
\begin{itemize}
\item[\rm (a)] $M(\la)$ has full-trimmed-Sylvester-rank if and only if $T_{1} (M)$ has full row rank.

\item[\rm (b)] Every $\widetilde{M} (\la) \in \FF[\la]^{m \times (m+n)}_{\underline{d}}$ such that
$
\|T_1 (M)- T_1 (\widetilde{M})\|_2 < \sigma_{\sum_{i=1}^m (d_i+1)} (T_{1}(M)) \,
$
has full-trimmed-Sylvester-rank.

\item[\rm (c)] There exists a polynomial matrix $\widetilde{M} (\la) \in \FF[\la]^{m \times (m+n)}_{\underline{d}}$ that does not have full-trimmed-Sylvester-rank and satisfies
$
\|T_1 (M)- T_1 (\widetilde{M})\|_2 = \sigma_{\sum_{i=1}^m (d_i+1)} (T_{1}(M)) \, .
$
\end{itemize}
\end{corollary}

\section{Perturbations of minimal bases dual to full-trimmed-Sylvester-rank matrices} \label{Sec:Dualminbases}

In this section we show that full-trimmed-Sylvester-rank polynomial matrices $M(\la)$ share an important property with the full-Sylvester-rank polynomial matrices described in \cite{VD}. As a consequence of Theorems \ref{thm.smoothfullSylvrank} and \ref{thm.propsfullSylvrank}, the row degrees of the minimal bases dual to full-trimmed-Sylvester-rank matrices remain constant (up to permutations) in a robustness neighborhood of $M(\la)$, and their values are given in \eqref{eq.statfullSylrank}. This allows us to show that one can always choose a basis for the perturbed dual space that varies smoothly with the perturbations of $M(\la)$, as long as the perturbations $M(\la) + \Delta M(\la)$ are restricted to stay in $\FF[\la]^{m \times (m+n)}_{\underline{d}}$ and one chooses correctly the degrees of freedom of the perturbed dual basis. We refer to \cite[Section 7]{VD} for a more elaborate discussion of these ideas in the context of $\FF[\la]^{m \times (m+n)}_{d}$ and we limit ourselves here to state Theorem \ref{thm.perturbdualbasis}. The proof of Theorem \ref{thm.perturbdualbasis} is similar to that of \cite[Theorem 7.1]{VD} except by some differences that we emphasize in the proof sketched below.

\begin{theorem} \label{thm.perturbdualbasis} Let $M(\la) \in \FF[\la]^{m \times (m+n)}_{\underline{d}}$, where $\underline{d} = (d_1,\ldots , d_m)$, be a polynomial matrix with full-trimmed-Sylvester-rank, let $k'$ and $t$ be defined as in \eqref{eq.ktprime}, and let $N(\la) \in \FF[\la]^{n \times (m+n)}_{k'}$ be a minimal basis dual to $M(\la)$ with highest-row-degree coefficient matrix $N_{hr}\in \FF^{n \times (m+n)}$. Moreover, let us define the quantities $\theta_1 (M)$ and $\theta_2 (M)$ as follows:
\begin{itemize}
\item[\rm (a)] If $k' >1$ and $t>0$
\begin{align*}
\theta_1 (M) & := \min \left\{
\frac{\sigma_{(k'-1)(m+n)} (T_{k' -1}(M))}{\sqrt{k'-1}} \, , \,
\frac{\sigma_{(k'm+\sum_{i=1}^m d_i)} (T_{k'}(M))}{\sqrt{k'}} \, , \,
\frac{\sigma_{(k'm+m+\sum_{i=1}^m d_i)} (T_{k'+1}(M))}{\sqrt{k'+1}}
\right\} \, , \\
\theta_2 (M) & := \min \left\{
\frac{\sigma_{(k'm+\sum_{i=1}^m d_i)} (T_{k'}(M))}{\sqrt{k'}} \, , \,
\frac{\sigma_{(k'm+m+\sum_{i=1}^m d_i)} (T_{k'+1}(M))}{\sqrt{k'+1}}
\right\} \, ;
\end{align*}

\item[\rm (b)] If $k' = 1$ and $t >0$,
\begin{align*}
\theta_1 (M) & = \theta_2 (M) := \min \left\{
\sigma_{(m+\sum_{i=1}^m d_i)} (T_{1}(M)) \, , \,
\frac{\sigma_{(2m+\sum_{i=1}^m d_i)} (T_{2}(M))}{\sqrt{2}}
\right\} \, ;
\end{align*}

\item[\rm (c)] If $t =0$
\begin{align*}
\theta_1 (M) & := \min \left\{
\frac{\sigma_{(k'm+\sum_{i=1}^m d_i)} (T_{k'}(M))}{\sqrt{k'}} \, , \,
\frac{\sigma_{(k'm+m+\sum_{i=1}^m d_i)} (T_{k'+1}(M))}{\sqrt{k'+1}}
\right\} , \\
\theta_2 (M) & :=
\frac{\sigma_{(k'm+m+\sum_{i=1}^m d_i)} (T_{k'+1}(M))}{\sqrt{k'+1}}
\, .
\end{align*}
\end{itemize}
Then, every $\widetilde{M} (\la) \in \FF[\la]^{m \times (m+n)}_{\underline{d}}$ such that
\begin{equation} \label{eq.dualpert1}
\|T_1 (M)- T_1 (\widetilde{M})\|_2 < \frac{1}{2} \cdot \theta_1 (M) \cdot  \frac{\sigma_n (N_{hr})}{\|S_1 (N)\|_F} \,
\end{equation}
has full-trimmed-Sylvester-rank and has a dual minimal basis $\widetilde{N} (\la) \in \FF[\la]^{n \times (m+n)}_{k'}$ that satisfies
\begin{equation} \label{eq.dualpert2}
\frac{\|S_1 (N) - S_1 (\widetilde{N})\|_F}{\|S_1 (N)\|_F} \leq \frac{2}{\theta_2 (M)} \cdot \|T_1 (M)- T_1 (\widetilde{M})\|_2 \, .
\end{equation}
In addition, if $t=0$, then all the row degrees of $\widetilde{N}(\la)$ and $N(\la)$ are equal to $k'$.
\end{theorem}

\begin{proof}
As said before, we only emphasize some differences with the proof of \cite[Theorem 7.1]{VD}. We invite the reader to follow the proof of \cite[Theorem 7.1]{VD} for the case $k' > 1$ and $t >0$ (the other cases are simpler), using the same notation, until the equations (7.6) and (7.7) in \cite{VD}, which are
\begin{align} \label{eq.3dualpartdegree}
S_{k'} (\widetilde{M}) \, S_1 (\Delta X^T) & = - S_{k'}(\Delta M) \,  S_1(X^T)\quad \mbox{and} \\ \label{eq.4dualpartdegree}
S_{k' +1} (\widetilde{M}) \, S_1 (\Delta Y^T) & = - S_{k'+1} (\Delta M) \, S_1(Y^T) \, ,
\end{align}
where the Sylvester matrices are defined assuming that $\widetilde{M}(\la)$ and $\Delta M (\la)$ have degree at most $d = \max_{1 \leq i \leq m} d_i$, $X (\la)$ and $\Delta X (\la)$ have degree at most $k' - 1$, and $Y (\la)$ and $\Delta Y (\la)$ have degree at most $k'$. The key point is that $S_{k'} (\widetilde{M})$ and $S_{k'}(\Delta M)$ (respectively, $S_{k' +1} (\widetilde{M})$ and $S_{k' +1}(\Delta M)$) have both some zero rows in the same positions that one can remove and obtain the trimmed Sylvester matrices $T_{k'} (\widetilde{M})$ and $T_{k'}(\Delta M)$ (respectively, $T_{k' +1} (\widetilde{M})$ and $T_{k' +1}(\Delta M)$). Therefore, \eqref{eq.3dualpartdegree}-\eqref{eq.4dualpartdegree} are equivalent to the following equations for the unknown polynomial matrices $\Delta X (\la)$ and $\Delta Y (\la)$
\begin{align} \label{eq.5dualpartdegree}
T_{k'} (\widetilde{M}) \, S_1 (\Delta X^T) & = - T_{k'}(\Delta M) \,  S_1(X^T)\quad \mbox{and} \\ \label{eq.6dualpartdegree}
T_{k' +1} (\widetilde{M}) \, S_1 (\Delta Y^T) & = - T_{k'+1} (\Delta M) \, S_1(Y^T) \, ,
\end{align}
which are consistent because $T_{k'} (\widetilde{M})$ and $T_{k'+1} (\widetilde{M})$ have both full row rank. From here, the proof is completely analogous to that of \cite[Theorem 7.1]{VD} and consists of bounding the minimum Frobenius norm solutions of \eqref{eq.5dualpartdegree} and \eqref{eq.6dualpartdegree}.
\end{proof}

\begin{remark} {\rm Note that, according to Theorem \ref{thm.propsfullSylvrank}-(a), the minimal bases $N(\la)$ and $\widetilde{N}(\la)$ dual to, respectively, $M(\la)$ and $\widetilde{M}(\la)$ appearing in Theorem \ref{thm.perturbdualbasis} have both $t$ row degrees equal to $k'-1$ and $n-t$ equal to $k'$. Therefore, if $t\ne 0$, we can order adequately the rows of $N(\la)$ and $\widetilde{N}(\la)$ and consider, without loss of generality, that $N(\la), \widetilde{N} (\la) \in \FF[\la]^{n \times (m+n)}_{\underline{k'}}$, where $$\underline{k'} = (\underbrace{k'-1, \ldots , k' -1}_{t}, \underbrace{k' , \ldots k'}_{n-t}) \,.$$
Then, we can use the trimmed Sylvester matrices of $N(\la), \widetilde{N} (\la) \in \FF[\la]^{n \times (m+n)}_{\underline{k'}}$ to express the results in Theorem \ref{thm.perturbdualbasis}, since the corresponding spectral and Frobenius norms are equal to those of the Sylvester matrices. More precisely, \eqref{eq.dualpert1} and \eqref{eq.dualpert2} can be written as
\[
\|T_1 (M)- T_1 (\widetilde{M})\|_2 < \frac{1}{2} \cdot \theta_1 (M) \cdot  \frac{\sigma_n (N_{hr})}{\|T_1 (N)\|_F} \,
\]
and
\[
\frac{\|T_1 (N) - T_1 (\widetilde{N})\|_F}{\|T_1 (N)\|_F} \leq \frac{2}{\theta_2 (M)} \cdot \|T_1 (M)- T_1 (\widetilde{M})\|_2 \, .
\]
However, we emphasize that, in general, $T_1 (M)$ and $T_1 (\widetilde{M})$ have different structures than $T_1 (N)$ and $T_1 (\widetilde{N})$, which might make the previous equations somewhat confusing.
}
\end{remark}

\section{On the classical rank conditions for robust minimal bases} \label{Sec:Classrevisited}

This section considers those minimal bases in $\CC[\la]^{m \times (m+n)}_{\underline{d}}$ that are robust under perturbations, which are those with full row rank leading row-wise coefficient matrix, according to Theorem \ref{thm.smoothminbases}. For these minimal bases, we prove that the infinitely many constant matrices whose ranks are involved in the classical Theorem \ref{minbasis_th} have minimum singular values bounded below by a common number determined by one of the trimmed Sylvester matrices of the considered minimal basis. This result generalizes to $\CC[\la]^{m \times (m+n)}_{\underline{d}}$ the result proved in \cite[Theorem 8.1]{VD} for polynomial matrices of degree at most $d$. In contrast with the results included in Sections \ref{Sec:FullRank}, \ref{Sec:Smoothness}, and \ref{Sec:Dualminbases}, the proof of Theorem \ref{thm.classrevisit} is more involved than the one of \cite[Theorem 8.1]{VD}, and, therefore, is fully included below. It is important to recall in the statement of Theorem \ref{thm.classrevisit} that, in order to avoid trivialities, we are assuming since Section \ref{Sec:Preliminaries} that $\max_{1 \leq i \leq m} d_i > 0$, which implies that $d' >0$ as a consequence of Theorem \ref{thm:dualbasis}.
\begin{theorem} \label{thm.classrevisit} Let $M(\la) \in \CC[\la]^{m \times (m+n)}_{\underline{d}}$ be a minimal basis with $\rank (M_{\underline{d}}) =m$, where $M_{\underline{d}}$ is the leading row-wise coefficient matrix of $M(\la)$ introduced in Definition \ref{def.Mbard} and $\underline{d} = (d_1 , \ldots , d_m)$. Let $d'$ be the largest right minimal index of $M(\la)$ and $T_{d'}$ be its $d'$th trimmed Sylvester matrix. Then
\[
\sigma_{(d'm+\sum_{i=1}^m d_i)}(T_{d'}) \leq  \inf_{\la_0\in \CC} \sigma_m (M(\la_0)) \quad \mbox{\rm and} \quad \sigma_{(d'm+\sum_{i=1}^m d_i)}(T_{d'}) \leq \sigma_m (M_{\underline{d}}) \,.
\]

\end{theorem}

\begin{proof}
From Corollary \ref{cor2:Mdfull}, we obtain that $T_{d'}$ has full row rank. Therefore, its smallest singular value is larger than zero, i.e., $\sigma_{(d'm+\sum_{i=1}^m d_i)}(T_{d'}) >0$. We use in this proof the well known fact \cite{HoJo, stewart-sun} that any matrix $A\in \CC^{p \times q}$ with $p \leq q$ satisfies
\begin{equation} \label{eq.minsingvalgen}
\sigma_p (A) = \min_{0 \ne x \in \CC^{p}} \frac{\|A^* x \|_2}{\|x\|_2} =
\min_{0 \ne x \in \CC^{p}} \frac{\|x^* A \|_2}{\|x\|_2} \, .
\end{equation}
This result applied to $A = T_{d'}$, together with Lemma \ref{lemm.nesttrim}, implies that $\sigma_{(d'm+\sum_{i=1}^m d_i)}(T_{d'}) \leq \sigma_m (M_{\underline{d}})$, since the block row $[0 \; M_{\underline{d}}]$ is a submatrix of $T_{d'}$ and one can choose in \eqref{eq.minsingvalgen} vectors $x$ with entries not corresponding to this submatrix equal to zero.
To prove the first inequality in Theorem \ref{thm.classrevisit}, we assume that $M(\la)$ is described through its rows as in \eqref{eq.Fdbar}. Then, note that the $i$th row $R_i(\la)\in \FF[\la]^{1\times (m+n)}$ of $M(\la)$ satisfies the following equality between polynomial matrices:
\begin{equation} \label{conv}\Pi_{d_i+d'}(\la) \, T_{d'}(R_i) =  R_i(\la) \, [ \Pi_{d'}(\la)\otimes I_{m+n}] \, ,
\end{equation}
where
$$ \Pi_k(\la) := \left[ \begin{array}{ccccc} 1 & \la & \la^2 & \ldots & \la^{k-1} \end{array}\right]
$$
and
$$  T_{d'}(R_i)=  \underbrace{\left[ \begin{array}{ccccc} R_{i,0} \\ R_{i,1} & R_{i,0} \\
   \vdots & R_{i,1} &  \ddots \\
   R_{i,d_i} & \vdots & \ddots & R_{i,0} \\
   0 & R_{i,d_i} & & R_{i,1}\\
   \vdots & \ddots & \ddots & \vdots \\
   0 & \ldots & 0 & R_{i,d_i}
   \end{array}\right]}_{d' \; \mathrm{blocks}}  \in \CC^{(d_i + d') \times d' (m+n)} \, .
$$
Theorem \ref{minbasis_th} implies that $\sigma_{0} := \sigma_m (M(\la_0)) >0$ for any $\la_0 \in \CC$. Let $u_0\in \CC^m$ and $v_0\in \CC^{(m+n)}$ be left and right singular vectors of $M(\la_0)$ corresponding to $\sigma_0$, that is $\|u_0\|_2=\|v_0\|_2=1$ and  $u_0^* \, M(\la_0)=\sigma_0 \, v_0^*$. Then it follows from \eqref{conv} that
\begin{equation} \label{eq.semilast}
u_0^*\left[ \begin{array}{cccc} \Pi_{d_1+d'}(\la_0) \\ & \ddots &\\  & & \Pi_{d_m+d'}(\la_0)  \end{array}\right]
  \left[ \begin{array}{c} T_{d'}(R_1) \\ \vdots \\ T_{d'}(R_m) \end{array}
\right] =
   \sigma_0(\left[ \begin{array}{ccccc} 1 & \la_0 & \la^2_0 & \ldots & \la^{d'-1}_0 \end{array}\right] \otimes v_0^*).
\end{equation}
Notice that the block arrangement with the matrices $T_{d'}(R_i)$ is nothing but a row permutation of $T_{d'}(M)$, and that the vector multiplying it on the left has 2-norm larger than or equal to  $\|\Pi_{d_s+d'}(\la_0)\|_2=\sqrt{\sum_{i=1}^{d_s+d'}|\la_0|^{2(i-1)}}$, where $d_s = \min_{1 \leq i \leq m} d_i$.
From \eqref{eq.semilast} and \eqref{eq.minsingvalgen} applied to the row permutation of $T_{d'} (M)$, we get
$$
\sigma_{(d'm+\sum_{i=1}^m d_i)}(T_{d'}) \le \sigma_0
\sqrt{\frac{\sum_{i=1}^{d'} |\la_0|^{2(i-1)} }{\sum_{i=1}^{d_s+d'}|\la_0|^{2(i-1)}}}\le \sigma_0 = \sigma_m (M(\la_0)).
$$
Since this holds for all $\la_0 \in \CC$, the result is proved.
\end{proof}

\section{Conclusions} \label{Sec:Conclusions}
In this paper we have extended the results previously obtained in \cite{VD} for the set of polynomial matrices with degree at most $d$, i.e., the set $\FF[\la]^{m\times (m+n)}_d$, to the set of polynomial matrices whose row degrees are at most $d_1, d_2, \ldots, d_m$, i.e., the set $\FF[\la]^{m\times (m+n)}_{\underline{d}}$, where $\underline{d} = (d_1, d_2, \ldots, d_m )$. In \cite{VD} we proved, among many other results, that generically the polynomial matrices in $\FF[\la]^{m\times (m+n)}_d$ are minimal bases with all its row degrees equal to $d$, i.e., with homogeneous row degrees, and with ``almost homogeneous'' right minimal indices, i.e., right minimal indices differing at most by one, determined by the constraint that their sum is equal to $md$. Analogously, we have shown in this paper that generically the polynomial matrices in $\FF[\la]^{m\times (m+n)}_{\underline{d}}$ are also minimal bases, in this case with their row degrees equal to $d_1, d_2, \ldots , d_m$, and again with ``almost homogeneous'' right minimal indices, which are determined now by the constraint that their sum is equal to $\sum_{i=1}^{m} d_i$. Thus, we have proved that the ``almost homogeneity'' of the right minimal indices is a general phenomenon that is independent of the values of the row degrees $d_1, d_2, \ldots , d_m$, which can be arbitrarily different, or, in other words arbitrarily inhomogeneous. Many other properties have been also extended from $\FF[\la]^{m\times (m+n)}_d$ to $\FF[\la]^{m\times (m+n)}_{\underline{d}}$ just by introducing minor changes to formulas, theorems, and proofs coming mainly from replacing the notion of Sylvester matrices by the new notion of trimmed Sylvester matrices. This allowed us to broaden the class of full-Sylvester-rank matrices to that of full-trimmed-Sylvester-rank matrices as a class of polynomial
matrices that are robust minimal bases and have a dual minimal basis with similar robustness properties. One important property that is {\em not preserved}
for this extended set is that its reversed polynomial matrix is not necessarily a minimal basis anymore. This last property is important when dealing with
so-called {\em strong} linearizations or $\ell$-ifications of polynomial matrices, but we expect that the extended set will play an important role for
problems where strongness is not an issue. Moreover, we emphasize that we are currently using some of the results in this paper for describing the sets of polynomial matrices with bounded rank and degree from a different perspective that the one recently introduced in \cite{dmy-dop-2017}, which will be more convenient in certain applications.

\bibliographystyle{plain}

\end{document}